%% file: SymplecticMain.tex
\documentclass[12pt]{amsart}
\pdfoutput=1

\usepackage[margin=1.75cm]{geometry}

\usepackage{amssymb,amsfonts}

\usepackage{enumerate}
\usepackage{mathrsfs}
\usepackage{hyperref}
\hypersetup{
   colorlinks,
    citecolor=black,
    filecolor=black,
    linkcolor=red,
   urlcolor=black
}
\usepackage{cleveref}
\usepackage{tikz-cd}
\usepackage{bbold}
\usepackage{marvosym}
\usepackage{color}    
\usepackage{graphicx}
\usepackage[final]{showlabels}
\usepackage[most]{tcolorbox}

\usepackage[mathscr]{eucal}

\newtheorem{Theorem}{Theorem}[subsection]
\newtheorem{theorem}[Theorem]{Theorem}
\newtheorem{Corollary}[Theorem]{Corollary}
\newtheorem{corollary}[Theorem]{Corollary}

\newtheorem{proposition}[Theorem]{Proposition}
\newtheorem{Lemma}[Theorem]{Lemma}
\newtheorem{lemma}[Theorem]{Lemma}

\theoremstyle{definition}
\newtheorem{Definition}[Theorem]{Definition}
\newtheorem{definition}[Theorem]{Definition}

\newtheorem{Example}[Theorem]{Example}
\newtheorem{example}[Theorem]{Example}

\newtheorem{Notations}[Theorem]{Notation}

\theoremstyle{remark}

\newtheorem{remark}[Theorem]{Remark}

\input{commands.tex}

\makeatletter
\let\c@equation\c@Theorem
\makeatother
\numberwithin{equation}{section}

\title{On the Essential Dimension of Symplectic Vector Bundles over Curves}

\author{Ajneet Dhillon}
\email{adhill3@uwo.ca}

\author{Sayantan Roy Chowdhury}
\email{sroycho4@uwo.ca}

\begin{document}
\maketitle

\begin{abstract}
   Let $X$ be a smooth geometrically connected projective curve of genus at least 2 over a field of 
   characteristic zero. We compute the essential dimension of the moduli stack of symplectic bundles
   over $X$. Unlike the case of vector bundles, we are able to precisely compute the essential dimension 
   as the generic gerbe of the moduli stack has period 2 over it's moduli space.  
\end{abstract}

\tableofcontents 

\input{introduction.tex}

\input{filteredSymplectic.tex}

\input{selfAdjoint.tex}


\input{nilStack.tex}



\input{equivalence.tex}





\input{mainResults.tex}

\end{document}

%% file: commands.tex
\DeclareMathOperator{\Spec}{Spec}

\DeclareMathOperator{\sheafHom}{\mathcal{H}om}
\DeclareMathOperator{\Hom}{Hom}
\DeclareMathOperator{\Homs}{Hom_s}

\DeclareMathOperator{\homsheaf}{\mathcal{H}om}
\DeclareMathOperator{\homs}{\mathcal{H}om_s}

\DeclareMathOperator{\image}{Im}

\DeclareMathOperator{\rank}{rank}
\DeclareMathOperator{\Bun}{Bun}
\DeclareMathOperator{\Sp}{Sp}

\DeclareMathOperator{\homsFz}{\mathcal{H}om_{s,F_0}}
\DeclareMathOperator{\coker}{coker}

\DeclareMathOperator{\Sch}{Sch} 
\DeclareMathOperator{\Sets}{Sets}

\DeclareMathOperator{\Cohomology}{H}

\DeclareMathOperator{\Aut}{Aut}
\DeclareMathOperator{\Sesq}{Sesq}
\DeclareMathOperator{\herm}{\mathfrak{herm}}
\DeclareMathOperator{\Gr}{gr}
\DeclareMathOperator{\proj}{\mathfrak{Proj}}
\DeclareMathOperator{\module}{\mathfrak{mod}}
\DeclareMathOperator{\can}{can}
\DeclareMathOperator{\Tr}{Tr}
\DeclareMathOperator{\ed}{ed}

\DeclareMathOperator{\N}{N}

\DeclareMathOperator{\Fgt}{Fgt}
\DeclareMathOperator{\trdeg}{trdeg}

\DeclareMathOperator{\Grass}{Grass}

\newcommand{\gr}[2][]{\operatorname{gr_{#2}^{#1}}}

\newcommand{\kerimage}[2]{\ker(\theta^{#1}) \cap \image(\theta^{#2})}

\newcommand{\bC}{{\bf C}}

\newcommand{\sO}{{\mathscr O}}
\newcommand{\sE}{{\mathscr E}}
\newcommand{\sF}{{\mathscr F}}
\newcommand{\sW}{{\mathscr W}}
\newcommand{\sV}{{\mathscr V}}
\newcommand{\sI}{{\mathscr I}}
\newcommand{\sJ}{{\mathscr J}}
\newcommand{\sM}{{\mathscr M}}
\newcommand{\sA}{{\mathscr A}}
\newcommand{\sG}{{\mathscr G}}

\newcommand{\cO}{\mathcal O}
\newcommand{\lkF}{\pi_{L}^{*}(\pi_K)_{*}\sF}

\newcommand{\shom}{{\mathcal Hom}}

\newcommand{\ZZ}{{\mathbb Z}}
\newcommand{\NN}{{\mathbb N}}

\newcommand{\fX}{{\mathfrak X}}

\DeclareMathOperator{\Fil}{Fil}
\DeclareMathOperator{\End}{End}
\DeclareMathOperator{\Nil}{Nil}
\DeclareMathOperator{\ev}{ev}

\newcommand{\tNil}{\widetilde{\Nil}}
\newcommand{\tN}{\widetilde{\N}}

\newcommand{\sEnd}{\mathcal{E}nd}
\newcommand{\Spr}{{\rm Sp}_{2r}}
\newcommand{\Coh}{\mathcal{C}oh}
\newcommand{\Vect}{\mathcal{V}ect}
\newcommand{\Nilscheme}{{\bf{Nil}}}

\newtcolorbox{important}{colback=red!5!white, 
colframe=red!75!black}

\newtcolorbox{todobox}{colback=green!5!white,
colframe=green!75!black}

\newtcolorbox{shaded}{colback=red!5!white, title={To be deleted},
colframe=red!75!black,  breakable}

%% file: introduction.tex
\section{Introduction}   

Let $k$ be a field of characteristic 0. Fix a smooth geometrically connected projective curve $X$ over $k$ of genus $g$ at least 2. 
In this paper we study the essential dimension of the moduli stack of principal bundles over the symplectic group $\Spr$. 
The main result of this paper is the calculation of the essential dimension, see \ref{t:essential} and \ref{c:ed symplectic}. 

A lower bound for the essential dimension can be obtained by combining \cite[7.3]{coskun} with \cite[5.4]{BRV}, see the proof 
of \ref{t:essential} for further details. The main content of this article is to prove the reverse inequality. For this we 
borrow an idea from \cite{BDH} and introduce in section \ref{s:nilpotent} an algebraic stack parameterising symplectic bundles with 
a nilpotent endomorphism. We are able to compute its dimension in \ref{Lemma Dimension count}. This ultimately leads to a bound on the transcendence 
degree of the field of moduli of a symplectic bundle in \ref{l:Field of Moduli for indecomposable bundles over algebraically closed field}. 
At this point we are reduced to studying the essential dimension of the residual gerbe of a symplectic bundle over it's field of moduli. 
To understand this number, we associate to a symplectic bundle an algebra with involution in \ref{s:categorical}. In this section it is shown that 
the essential dimension of this residual gerbe can be computed studying the essential dimension of Hermitian modules over this algebra. 
Using elementary means we can bound the essential dimension of the functor of Hermitian modules. 
An unpublished result of  Danny Ofek  and Zinovy Reichstein along with \cite{lotscher} provides a precise calculation of this number. 

A broad overview of essential dimension can be found in \cite{reichstein12, reichstein10}. 
Given an algebraic stack  $\mathfrak{X}$ over a field $k$ the essential dimension of  $\mathfrak{X}$  is a measure of the number of parameters 
needed to parameterise a generic element of  $\mathfrak{X}$. This number first appeared in the work of Reichstein and Buhler \cite{buhler}.
From another point of view there are two ways to define the dimension of a variety. The first is the Krull dimension and the second via the transcendence 
degree of it's function field. The function field approach when generalised to algebraic stacks yields the notion of essential dimension, see \cite{BRV}. 
To make this precise consider a functor 
$$
F: \text{Fields}_{k}\to \text{Sets}
$$
where $\text{Fields}_{k}$ is the category of field extensions of $k$. Given a field extension $K/k$ and $x\in F(K)$ we say that a subfield 
$K\supseteq L \supseteq k$ is a \emph{field of definition} for $x$ if there is a $y\in F(L)$ with $x=F(i)(y)$ where $i:L\hookrightarrow K$ is the inclusion.   
The \emph{essential dimension} of $x$ is defined to be 
$$
\ed_{k}(x) = \min \text{tr.deg}(L/k)
$$
where the minimum is taken over all fields of definition of $x$. The \emph{essential dimension of $F$} is defined to be 
$$
\ed_{k}F =\sup \ed_{k}(x)
$$
where the supremum is taken over all $x$ and all field extensions $K/k$. 
If  $\mathfrak{X}$ is an algebraic stack over a field $k$ then the \emph{essential dimension of  $\mathfrak{X}$ }, denote $\ed_{k} \mathfrak{X}$ is defined to 
be the essential dimension of the associated functor of isomorphism classes of objects over fields.

The main result of this paper is a computation of $\ed_k \Bun_{X,\Spr}$ when $X$ has genus at least 2  and $\Bun_{X,\Spr}$ is the moduli stack 
of $\Spr$-bundles on $X$. We prove:

\begin{theorem}
    Let $X$ be a curve of genus $g \ge 2$ over a field $k$ of characteristic 0. Then 
\[      \ed_k(\Bun_{X,\Spr}) = (g-1)( \dim \Spr) + 2^{\nu_2(2r)}          \] 
where $\nu_2(2r)$ is the highest power of 2 dividing $2r$.  
\end{theorem}

The coarse moduli space of $\Bun_{X,\Spr}$ has dimension $(g-1)( \dim \Spr)$. The term $2^{\nu_2(2r)}$ is the extra contribution to the essential 
dimension coming from the generic gerbe of $\Bun_{X,\Spr}$ over it's coarse moduli space. This is a $\mu_{2}$-gerbe. As the main conjecture in 
\cite{ct} is known for $\mu_{2}$-gerbes we are able to compute the essential dimension, in contrast to the vector bundle case in \cite{BDH}. 

The paper begins in section 2 where we introduce the notions of filtered and graded symplectic vector bundles. This notion arises naturally 
later in the paper as a symplectic vector bundle with a nilpotent self-adjoint endomorphism has a natural filtration. In section 3
we study sheaves of homomorphism associated to filtered symplectic vector bundles. These properties will be needed in section 4.7 to compute 
the dimension of the stack $\Nil_{X,\Spr}^{n}$. In section 4, we introduce the main stacks used in the paper and study their deformation theory and 
dimensions. The main result in this section is \ref{Lemma Dimension count}. In section 5 we after recalling some notions pertaining to modules 
over a ring with involution we prove the main result of the section, \ref{l:cat equivalence}. This result describes a categorical equivalence
between modules and objects of the residual gerbe  of a point in $\Bun_{X,\Spr}$. This result forms a bridge between principal bundles 
and the theory of rings with involution. 

An underlying assumption in this paper is that $char(k)=0$. This fact is used in section 5, in order for the trace pairing to be non-degenerate. 
It is also applied in  
 Lemma \ref{l:surjectivity on associated graded vector space case} 
    and Lemma \ref{l:Aut unipotent}.

\section*{Acknowledgements}

Both authors would like to thank Danny Ofek and Zinovy Reichstein for valuable input. 

Ajneet Dhillon was partially supported by an NSERC discovery grant.

\section*{Notation} 

\begin{tabular}{cl}
    $k$ & Our base field of characteristic 0. \\ 
    $X$ & A smooth projective, geometrically connected curve over $k$. \\ 
    $\ed_{K}{\mathfrak X}$ & The essential dimension of the stack or functor ${\mathfrak X}$ over the field $K$.  \\ 
    & The reader is referred to the introduction  for a precise definition \\ 
    & Basic properties of essential dimension can be found in \cite{BRV} or \cite{BDH}.\\
    $\sF^{\vee}$ & The dual sheaf, $\shom(\sF,\sO_{S})$ of the quasi-coherent sheaf $\sF$ on a scheme $S$. \\ 
    $\Bun_{X,\Spr}$ & The moduli stack principal $\Spr$-bundles over $X$, see \ref{d:BunSp}. \\ 
    $\tNil_{X,\Spr}^n$ & The stack of self-adjoint nilpotent endomorphisms, see \ref{d:nilEnd}. \\ 
    $\Nil_{X,\Spr}^{n}$ & The stack of filtered self-adjoint endomorphisms, see \ref{d:nilStack}. \\ 
    $M^{*}$ & The dual of a module with respect to a ring with involution, see \ref{d:M to M*}.  \\
    $\proj(A)$ & The category of finitely generated projective modules over a ring $A$. \\
    $\herm_{\epsilon}(A,\mu)$ & The category of $\epsilon$-Hermitian modules over a ring with involution, see \ref{d:herm}. \\
\end{tabular}

%% file: filteredSymplectic.tex
\section{Filtered symplectic vector bundles}

 
\subsection{Graded vector bundles.}

Let $S$ be a scheme. \emph{A graded vector bundle} on $S$ is a locally free sheaf $\sE$ of finite rank on $S$ together with a direct 
sum decomposition 
$$
\sE = \bigoplus_{i\in\ZZ} \sE_{i}. 
$$
Necessarily each of the $\sE_{i}$ are locally free sheaves of finite rank and if $S$ is quasi-compact, almost all of them are zero. 

A morphism of graded vector bundles is defined in the obvious way. We denote by ${\rm Gr}(S)$ the category of graded vector bundles on $S$. 

The dual of a graded vector bundle is graded in the following way:
$$
(\sE)^{\vee} = \bigoplus_{i\in\ZZ} (\sE^{\vee})_{i}\quad \text{where}\quad (\sE^{\vee})_{i}= (\sE_{-i})^{\vee}. 
$$

If $\sE=\bigoplus_{i\in\ZZ} \sE_{i}$ is a graded vector bundle and $n\in \ZZ$ then we can form a new graded vector bundle 
$\sE[n]$ with shifted grading given by 
$$
(\sE[n])_{i}= \sE_{i+n}. 
$$
If $\sE=\bigoplus_{i}\sE_{i}$ and $\sF=\bigoplus_{i}\sF_{i}$ are graded vector bundles then a morphism of vector bundles 
$\theta:\sE\to \sF$ is said to be graded of degree $a$, where $a\in \ZZ$, if $\theta=\bigoplus\theta_{i}$ where 
$$
\theta_{i}:\sE_{i}\to \sF_{i+a}. 
$$

\subsection{Filtered vector bundles.}

Let $S$ be a scheme.
The category of vector bundles on $S$  has the structure of an exact category, the admissible monomorphisms are monomorphisms whose cokernel 
is a vector bundle.

A \emph{filtration}  $F$ on a vector bundle  $\sE$ consists of a sequence of admissible subobjects
$$
F^{n}(\sE)\hookrightarrow F^{n+1}(\sE)\hookrightarrow \ldots \hookrightarrow F^{m-1}(\sE)\hookrightarrow F^{m}(\sE)=\sE
$$
where $n<m\in \ZZ$. The fact that the inclusions are admissible means that 
$$
\Gr_{F}^{i}(\sE):= F^{i}(\sE)/F^{i-1}(\sE)
$$
is locally free. 
If $F^{n}(\sE)\ne 0$ and $F^{m-1}(\sE)\ne \sE$ we will say that the \emph{amplitude} of the filtration is $[n,m]$. We will sometimes 
denote this data by $(\sE,F)$. It will be convenient to extend the filtration outside the range $[n,m]$. This is done by setting 
$$
F^{i}(\sE) = 
\begin{cases}
    \sE & i>m \\
    0 & i<n. 
\end{cases}
$$

There is an obvious notion of a morphism of filtered vector bundles. We denote by $\Fil(S)$ the category of filtered vector bundles on $S$. 

If $n\in \ZZ$ then there 
is a shifting by $n$ functor $\Fil(\bC)\to \Fil(\bC)$ that shifts filtrations by $n$. We denote the shifted filtration by 
$(\sE,F[n])$, and we have 
$$
F^{i}[n](\sE) = F^{i+n}(\sE). 
$$

If $(\sE,F)$ is a filtered object then its dual is also filtered with filtration $(\sE^{\vee},F^{\vee})$. We define 
$$
F^{\vee, -i}(\sE^{\vee}) = \ker(\sE^{\vee}\to F^{i}(\sE)^{\vee}). 
$$
The minus sign is needed to ensure that the filtration is increasing. 

If $(\sE,F)$ is a filtered object we let 
$$
\Gr_{F}^{i}(\sE) = F^{i}(\sE) /F^{i-1}(\sE) \quad\text{and}\quad \Gr_{F}(\sE) = \bigoplus_{i} \Gr_{F}^{i}(\sE). 
$$
We will call $\Gr_{F}(\sE)$ the associated graded vector bundle  of $(\sE,F)$. 

There is an obvious functor 
$$
\Gr : \Fil(S)\to {\rm Gr}(S). 
$$
This functor is not compatible with taking duals. Instead, we have a canonical isomorphism
\begin{equation}\label{e:filtshift}
    \left(\Gr(\sE,F)\right)^{\vee}[-1] \cong \Gr((\sE^{\vee},F^{\vee})). 
\end{equation}

We can equip $\Gr_{F}(\sE)$ with a filtration defined by 
$$
F^{i}(\Gr_{F}(\sE)) = \bigoplus_{j\le i} \Gr_{F}^{i}(\sE). 
$$
We say that a filtered object $(\sE,F)$ is \emph{split filtered} if
there is an isomorphism of filtered objects
$$
\phi: (\sE,F) \stackrel{\sim }{\to } \Gr_{F}(\sE)
$$
satisfying $\phi_{j}(x) = \bar{x}$ for $j\ge i$
where $x\in F^{i}(\sE)$ and $\phi_{j}$ is the composition 
$$
(\sE,F) \stackrel{\sim }{\to } \Gr_{F}(\sE) \stackrel{proj }{\to } \Gr_{F}^{j}(\sE). 
$$

A choice of such an  isomorphism $\sE\to \Gr_{F}(\sE)$ will be called a \emph{ splitting}. 

\begin{Example}
    If $A$ is a ring then consider $\bC=$ the category of finitely generated projective $A$-modules. In this category every 
    filtered object is split filtered.  
\end{Example}

 
\subsection{Symplectic vector bundles.}

Let $S$ be a scheme. 
 A \emph{symplectic vector bundle} is a vector bundle $\sE$ on $S$ together with an isomorphism
$$
h:\sE\to \sE^{\vee}
$$
such that the following diagram commutes
\begin{center}
    \begin{tikzcd}
        \sE\ar[r,"h"] \ar[d,"\ev"] & \sE^\vee\ar[d,"-1"] \\ 
        \sE^{\vee\vee} \ar[r,"h^\vee"] & \sE^{\vee}. 
    \end{tikzcd}
\end{center}
The left vertical arrow is the canonical isomorphism. 
The data of $h$ is the same as that of a bilinear form 
$$
Q:\sE\otimes \sE\to \sO_{S}
$$
that is non-degenerate and alternating. We will summarize this data by saying that 
$
(\sE, h)$
is a symplectic vector bundle. 

If $\sW\hookrightarrow \sE$ is a subsheaf, not necessarily admissible, we define
$$
\sW^{\perp}:= \ker(\sE \xrightarrow{h} \sE ^{\vee}\to \sW^{\vee}). 
$$

\begin{Lemma}
    \label{l:perp splitting}
    Suppose that the inclusion $\sW\hookrightarrow \sE$ is split. Then the inclusion 
    $$
    \sW^{\perp}\hookrightarrow \sE
    $$
    is split. 
\end{Lemma}

\begin{proof}
    We have an exact sequence 
    $$
    0\to \sW^{\perp}\to \sE \to \sW^{\vee}\to 0. 
    $$
    As $\sW\hookrightarrow \sE$ is split so is the map $\sE^{\vee}\twoheadrightarrow \sW^{\vee}$. 
\end{proof}

It follows that if $\sW\hookrightarrow \sE$ is admissible then so is $\sW^{\perp}\hookrightarrow \sE$.

A morphism of symplectic vector bundles $f:(\sE,h)\to (\sF,g)$ is a morphism of vector bundles $f:\sE\to \sF$ such that the following diagram commutes:
\begin{center}
    \begin{tikzcd}
        \sE \ar[r,"h"]\ar[d,"f"] & \sE^\vee \\ 
        \sF \ar[r,"g"] & \sF^\vee \ar[u,"f^\vee"].\\    \end{tikzcd}
\end{center}

It follows easily that such 
a morphism is always a monomorphism of vector bundles. Hence, if $f:(\sE, Q)\to (\sF, Q') $ is morphism 
and $\sE\to \sF$ is admissible and $\rank(\sE)=\rank(\sF)$ then it is necessarily an isomorphism. 

Suppose that $(\sE, h)$ is a symplectic vector bundle 
If $f:\sE\to \sE$ is an endomorphism of the underlying vector bundle $\sE$, not necessarily a symplectic endomorphism, then we define
$$
\mu(f) = h^{-1}\circ f^{\vee} \circ h.
$$
We say that $f$ is \emph{self-adjoint }if $\mu(f)=-f$. 

Suppose that $(\sE,h)$ is a symplectic vector bundle and $\sE=\bigoplus \sE_{i}$ is graded. If the 
morphism 
$$
h:\sE\to \sE^{\vee}
$$
is a morphism of graded vector bundles then we will say that $(\sE=\oplus \sE_{i},h)$ is a \emph{graded symplectic vector bundle}.  
A graded endomorphism $\theta:\sE\to \sE$ is said to be self adjoint if $\mu(\theta)=-\theta $. 

\begin{lemma}\label{l:gradedSymplectic}
    Let $(\sE=\bigoplus \sE_{i},h)$ be a graded symplectic vector bundle. Then 
    $h$ decomposes as $h=\oplus_{i} h_{i}$ where 
    $$ h_{i}:\sE_{i}\to \sE_{-i}^{\vee}$$
    and $h_{i}^{\vee}=-h_{-i}$. 
    In this setting a graded morphism of degree $a$, say $\theta:\sE\to\sE$ is self adjoint if it decomposes as 
    $$
    \theta =\bigoplus_{i}\theta_{i}\quad\text{where}\quad 
    \theta_{i}:\sE_{i}\to \sE_{i+a}
    $$
    and 
    $$
    -\theta_{i}= h_{i+a}^{-1}\circ \theta_{-i-a}^{\vee}\circ h_{i}. 
    $$
    
\end{lemma}

\begin{proof}
  The decomposition follows directly from the definition of the induced grading on $\sE^{\vee}$. The second assertion
  follows from the diagram defining the notion of symplectic vector bundle. 
\end{proof}

Suppose that $(\sE,h)$ is a symplectic vector bundle and $(\sE,F)$ is a filtration. 
Recall that $\sE^{\vee}$ inherits a filtration from $\sE$. We say that $(\sE,h,F)$ is a
\emph{filtered symplectic vector bundle} if
$$
h:(\sE,F)\to (\sE^{\vee}, F^{\vee}[1])
$$
is a morphism of filtered objects. Note that is necessarily an isomorphism.

We will later produce examples of filtered 
symplectic objects obtained from symplectic sheaves equipped with nilpotent 
self adjoint endomorphisms. The presence of the 1 in the definition is justified
by appealing to these examples, also see Lemma \ref{l:sympFiltDual} below. 
To understand what this condition is asserting in terms of $\sE$, we have
$$
F^{i}(\sE)^{\perp}= h^{-1}(F^{\vee,-i}(\sE^{\vee})). 
$$
Then the condition becomes 
$$
F^{-1-i}(\sE)=F^{i}(\sE)^{\perp} \quad\text{or}\quad F^{j}(\sE)=F^{-j-1}(\sE)^{\perp}. 
$$

\begin{Lemma}
    \label{l:filtSympConstruct}
    Consider a collection of vector bundles $\sE_{i}$ with $-n\le i \le n$ and $n\in \NN^{>0}$.
    Suppose that we have isomorphisms $h_{i}:\sE_{i}\to \sE_{-i}^{\vee}$ for $-n\le i \le n$. 
    
    If $h_{i}=-h_{-i}^{\vee}$ then 
    $$
    \bigoplus_{i=-n}^{n}h_{i}:\bigoplus_{i=-n}^{n} \sE_{i}\to \bigoplus_{i=-n}^{n} \sE_{i}^{\vee}
    $$
    is a filtered symplectic object with filtration given by 
    $F^{k}(\bigoplus_i \sE_{i}) = \bigoplus_{i\le k} \sE_{i}$. Furthermore, this filtered object is split. 

    Finally, consider morphisms 
    $$
    \theta_{i,j}:\sE_{i}\to \sE_{j}.  
    $$
    Then $\theta = \bigoplus_{i,j}\theta_{i,j}$ is self adjoint if and only if 
    $$
    \theta_{i,j} = -h_{j}^{-1}\theta_{-j,-i}^{\vee}h_{i}. 
    $$
\end{Lemma}

\begin{proof}
    Let $\sE=\bigoplus \sE_{i}$ and $h=\bigoplus h_{i}$. The given map is clearly an isomorphism. 
    It is easy to check that the morphism satisfies the required property to induce a symplectic object. 
    The filtration is clear split, so it just remains to check that $h$ preserves filtrations. 
    We compute to see that 
    $$
    F^{\vee,-j}(\sE^{\vee}) = \bigoplus_{i=j+1}^{n}\sE_{i}^{\vee}\quad\text{and}\quad F^{i}(\sE)^{\perp}=\bigoplus_{i=-n}^{-1-j} \sE_{i}.
    $$
    It follows that the object is filtered symplectic. It is straightforward to see that the filtration is split. 

    The assertion about self adjointness of $\theta$ follows readily from a computation. 
\end{proof}

A filtered symplectic vector bundle isomorphic to one constructed in the proposition is called a \emph{split filtered symplectic vector bundle.}

Suppose that $(\sE,h,F)$ is a symplectic filtered sheaf or module. Then we have exact sequences
$$
0\to F^{i}(\sE)\to \sE^{\vee} \to (F^{-i-1}(\sE))^{\vee}\to 0. 
$$
where the first map is induced by $h$.
In other words there are induced isomorphisms 
$F^{i}(\sE) \cong (\sE/F^{-1-i}(\sE))^\vee$ as the dual is exact. 
Hence 
\begin{eqnarray*}
    F^{i}(\sE)/F^{i-1}(\sE) &\cong & \coker( (X/F^{-i}(\sE))^\vee \to (X/F^{-1-i}(\sE))^\vee ) \\
    &\cong & (F^{-i}(\sE)/F^{-i-1}(\sE))^\vee. 
\end{eqnarray*}

In summary:
\begin{Lemma}\label{l:sympFiltDual}
    Suppose that $(\sE,h,F)$ is a filtered symplectic sheaf or module.  Then there are induced isomorphisms
    $$
    h_{i}:\Gr_{F}^{i}(\sE)\cong (\Gr_{F}^{-i}(\sE))^{\vee}
    $$
    for all $i$. Furthermore, these isomorphisms satisfy the conditions of Lemma \ref{l:filtSympConstruct} so that 
    $\Gr_{F}(\sE)$ inherits the structure of a split filtered, symplectic vector bundle. 
    
\end{Lemma}

\begin{proof}
    The first part is the calculation above. The second part follows from the fact that $h$ is symplectic. 
    
\end{proof}

 
\subsection{The local structure of filtered symplectic vector bundles.}

\begin{Lemma}
    \label{l:easySplitting}
    Suppose that we have a split inclusion
    \[    \sW \oplus \sV \hookrightarrow \sE          \] 
    in the category $\bC$. 
    Then there exists a splitting of the map 
    \[     \sE^{\vee} \twoheadrightarrow  \sW^{\vee}                                \]
    such that every element in the image of this splitting vanishes on $\sV$.
    \end{Lemma}
    
    \begin{proof}
    This is elementary. 
    \end{proof}

    If $(\sE,h)$ is a symplectic vector bundle then a subsheaf $\sW\subseteq \sE$ is said to be \emph{isotropic} if $\sW\subseteq \sW^{\perp}$.

    \begin{Lemma}
        \label{l:twoStepSplitting}
        Let $(\sE,Q)$ be a symplectic vector bundle on a scheme $S$. Consider an isotropic locally free subsheaf
        \[  0 \subset \sW  \subset \sE                                                       \]
        such that both the inclusions $\sW\hookrightarrow \sW^{\perp}$ and $\sW^{\perp}\hookrightarrow \sE$ are 
        admissible. 
        Then the filtration
        \[ 
        0\subseteq F^{-1}(\sE)=\sW \subseteq F^{0}(\sE)=\sW^{\perp} \subseteq F^{1}(\sE)=\sE
        \]
        forms a filtered symplectic vector bundle $(\sE,Q,F)$. 
    
        Suppose now that $S$ is affine. If a splitting of the filtration on $\sW^{\perp} $ is chosen then 
        there exists a splitting of the filtration on $\sE$ extending the splitting of the filtration on $\sW^{\perp}$ such
        that the map 
        \[     \sE/\sW^{\perp} \hookrightarrow \sE \xrightarrow{h} \sE^{\vee} \twoheadrightarrow (\sW^{\perp}/\sW)^{\vee}                                                     \]
    vanishes.
    \end{Lemma}
        
    \begin{proof}
        We need to show that we have a diagram
        \begin{center}
            \begin{tikzcd}
                0\arrow[hookrightarrow]{r} & \sW\arrow[hookrightarrow]{r}\arrow[d] & \sW^\perp\arrow[hookrightarrow]{r} \arrow[d] & \sE \arrow[d]\\
                0\arrow[hookrightarrow]{r} & \left(\sE/\sW^\perp \right)^\vee\arrow[hookrightarrow]{r} & \left(\sE/\sW\right)^\vee \arrow[hookrightarrow]{r} & \sE^\vee \\
            \end{tikzcd}
        \end{center}
        where the vertical arrows are isomorphisms. The right two arrows are clearly isomorphisms. 
        To complete the proof we need to see that the inclusion $\sW\subseteq \sW^{\perp\perp}$ is an equality. As $\sW^{\perp}/\sW$ is locally free
        we just need to compute ranks. But we always have $\rank T^{\perp} = \rank \sE -\rank T$ for a locally free subsheaf $T\subseteq \sE$. This 
        settles the first claim.

        Assume now that $S$ is affine. 
        Since $\sE/\sW^{\perp}$ is  projective, the map
        \[      \sW \oplus \sW^{\perp}/\sW \stackrel{\sim}{\to} \sW^{\perp}  \hookrightarrow  \sE        \]
        is split. Write $i:\sW\hookrightarrow \sE$ for the inclusion. By Lemma \ref{l:twoStepSplitting}, we can 
        choose a splitting  $\sigma^{\vee}$ of 
        $i^{\vee}:\sE^{\vee} \twoheadrightarrow \sW^{\vee}$ that vanishes on $\sW^{\perp}/\sW$.
        We observe that the sequence 
        $$
        0\to \sW^{\perp}\to \sE \stackrel{h^{-1}i^{\vee}h}{\longrightarrow }\sE/\sW^{\perp}\to 0
        $$
        is in fact exact and hence the composition 
        $$
        \sE/\sW^{\perp}\to \sW^{\vee} \stackrel{\sigma^{\vee} }{\to } \sE^{\vee} \to \sE
        $$
        is a splitting of the inclusion $\sW^{\perp}\hookrightarrow \sE$. 
        The result follows from a diagram chase. 
    
    \end{proof}

    \begin{Lemma}\label{l:twoStepVanishing}
         Suppose that $S$ is an affine scheme and
         $(\sE,h)$ is a symplectic vector bundle on $S$. Let 
        \[  0 \subset \sW \subset  \sW^{\perp} \subset \sE                                                       \]
        be as in the previous Lemma and assume further that
         the vector bundle $\sE/\sW^{\perp}$ is  free. 
         Suppose that we have a chosen a splitting of $\sW^\perp \twoheadrightarrow \sW^\perp/\sW$. 
        Then there exists a splitting of the filtration on $\sE$ extending the chosen splitting of the filtration on $\sW^{\perp}$ such
        that both the maps
        \[     \sE/\sW^{\perp} \hookrightarrow \sE \xrightarrow{h} \sE^{\vee} \twoheadrightarrow (\sW^{\perp}/\sW)^{\vee}                                                     \]
        \[     \sE/\sW^{\perp} \hookrightarrow \sE \xrightarrow{h} \sE^{\vee} \twoheadrightarrow (\sE/\sW^{\perp})^{\vee}                                                     \]
    vanish.
    \end{Lemma}
    
    \begin{proof}
    By Lemma \ref{l:twoStepSplitting} there is an extension of the 
    splitting of filtration on $\sW^{\perp}$ to a splitting of the filtration of $\sE$ 
    such that the map  
    \[     \sE/\sW^{\perp} \hookrightarrow \sE \xrightarrow{h} \sE^{\vee} \twoheadrightarrow (\sW^{\perp}/\sW)^{\vee}     \]
    vanishes.
    
    Since $\sE/\sW^{\perp}$ is free, we can choose a basis $\{f_1,f_2 \hdots f_s\}$ of $\sE/\sW^{\perp}$. The dual basis of 
    $(\sE/\sW^{\perp})^{\vee}$ is denoted $\{f_1^{\vee},f_2^{\vee} \hdots f_s^{\vee}\}$. 
    
    We define elements $e_i$ of $\sW$ as 
    \[  e_i = h^{-1} (f_i^{\vee})                                      \] 
    for $1 \le i \le s$. Note that $e_{i}\in \sW$ as $\sE$ is filtered symplectic by the previous Lemma. Let $Q$ be the bilinear 
    form associated to $h$. 
    We now define a map $\sE/\sW^{\perp} \hookrightarrow \sE$  by sending elements $f_i$ to $\overline{f}_i$ for $1 \le i \le s$ 
    with $\overline{f}_i$ defined as 
    \begin{align*}
        \overline{f}_1 &= f_1 \\ 
        \overline{f}_2 &= f_2 - Q(f_2,f_1)e_1 \\  
        \overline{f}_3 &= f_3 - Q(f_3,f_1)e_1 - Q(f_3,f_2)e_2 \\ 
            &\vdots \\
        \overline{f}_i &= f_i - \sum_{k=1}^{i-1} Q(f_3,f_k)e_k  \\ 
        &\vdots \\ 
        \overline{f}_s &= f_s - \sum_{k=1}^{s-1} Q(f_3,f_s)e_s.
    \end{align*}
    Since for all $i$, the $e_i$'s lie in $\sW$, $Q(\overline{f}_i , x) = Q(f_i, x)$ for all elements $x$ in $\sW^{\perp}$, $1 \le i \le s$.
    the map  
    \[     \sE/\sW^{\perp} \hookrightarrow \sE \xrightarrow{h} \sE^{\vee} \twoheadrightarrow (\sW^{\perp}/\sW)^{\vee}     \]
    is still 0. Moreover, it follows from computation $Q(\overline{f}_i,\overline{f}_j) = 0$ for all $1 \le i,j \le s$. Hence, 
    the map 
    \[     \sE/\sW^{\perp} \hookrightarrow \sE \xrightarrow{h} \sE^{\vee} \twoheadrightarrow (\sE/\sW^{\perp})^{\vee}                                                     \]
    also vanishes.
    \end{proof}

    \begin{Lemma}
    \label{l:splitFilteredSympTwo}
    Let $(\sE,h)$ be a symplectic vector bundle on a scheme $S$. Consider an admissible isotropic
    subbundle $\sW\hookrightarrow \sE$ with induced filtered, symplectic vector bundle
    \[  0 \subset \sW=F^{-1}(\sE) \subset  \sW^{\perp}=F^{0}(\sE) \subset \sE=F^{1}(\sE),       \]
    see \ref{l:twoStepSplitting}. 
    If $\sE/\sW^{\perp}$ is free and the filtration  on $\sW^{\perp}$ is split, 
    then  $(\sE,h,F)$ is a split, filtered, symplectic vector bundle. 
    \end{Lemma}
    
    Recall that the term ``split, filtered, symplectic vector bundle '' was defined in the discussion immediately following 
    Lemma \ref{l:filtSympConstruct}.

    \begin{proof}
    We choose a splitting of the filtration  $E$ as in Lemma \ref{l:twoStepVanishing}. 
    The maps 
\begin{eqnarray*}
    \sW &\hookrightarrow \sE \xrightarrow{h} \sE^{\vee} &\twoheadrightarrow \sW^{\vee} \\
    \sW & \hookrightarrow \sE \xrightarrow{h} \sE^{\vee} &\twoheadrightarrow (\sW^{\perp}/\sW)^{\vee} \\
    \sW^{\perp}/\sW &\hookrightarrow \sE \xrightarrow{h} \sE^{\vee} &\twoheadrightarrow \sW^{\vee}
\end{eqnarray*}
    vanish since $\sE$ is a filtered symplectic object. Lemma \ref{l:twoStepVanishing} shows
    that the maps 
    \begin{eqnarray*}
        \sE/\sW^{\perp} &\hookrightarrow \sE \xrightarrow{h} \sE^{\vee} &\twoheadrightarrow (\sE/\sW^{\perp})^{\vee} \\
        \sE/\sW^{\perp} &\hookrightarrow \sE \xrightarrow{h} \sE^{\vee} &\twoheadrightarrow (\sW^{\perp}/\sW)^{\vee}     
    \end{eqnarray*}
    vanish as well for a particular splitting. The map 
    \[ \sW^{\perp}/\sW \hookrightarrow \sE \xrightarrow{h_Q} E^{\vee} \twoheadrightarrow (\sE/\sW^{\perp})^{\vee} \]
    vanishes since it is dual of the last map. It is easily checked that this splitting of the filtration gives $(\sE,h,F)$ the structure of a filtered split,
    symplectic vector bundle.
    \end{proof}

    \begin{Lemma}
            \label{l:inductionStep}
    Let $(\sE,h)$ be a filtered symplectic vector bundle equipped with filtration 
        \[    0 = F^{-n}(\sE) \subset F^{-n+1}(\sE)  \subset \hdots \subset F^{n-1}(\sE) =   \sE     \]
        such that $\sE/F^{n-2}(E)$ is free and $n\ge 2$. 
    Then the following statements hold. 
    \begin{enumerate}
        \item The vector bundle $F^{1-n}(\sE)$ is isotropic with $F^{1-n}(\sE)^{\perp} =  F^{n-2}(\sE)$.
        \item There are induced isomorphisms 
        $$
        h_{1-n}:F^{1-n}(\sE)\to (\sE/F^{n-2}(\sE))^{\vee}\quad\text{and}\quad 
        h_{n-1}:\sE/F^{n-2}(\sE)\to F^{1-n}(\sE)^{\vee}. 
        $$
        \item There is an induced isomorphism 
        $$
        \bar{h}:F^{n-2}(\sE)/F^{1-n}(\sE) \to (F^{n-2}(\sE)/F^{1-n}(\sE))^{\vee}
        $$
        that makes it into a filtered symplectic vector bundle where the filtration is induced from the filtration 
        on $F^{n-2}(\sE)$. 
        \item There is an isomorphism of filtered symplectic vector bundles 
        $$
        (\sE,F,h)\cong (F^{1-n}(\sE)\oplus F^{n-2}(\sE)/F^{1-n}(\sE))\oplus (\sE/F^{n-2}(\sE)), F', h_{1-n}\oplus \bar{h}\oplus h_{n-1} 
        $$
        where the filtration $F'$ has bottom term $F^{1-n}(\sE)$, middle terms induced from the filtration on $F^{n-2}(\sE)$. 
    \end{enumerate}
    \end{Lemma}
    
    \begin{proof}
    The first part follows from the definition of a filtered symplectic vector bundle and the fact that $n\ge 2$.
    The second and third part follow from Lemma \ref{l:splitFilteredSympTwo}. The fact that $h_{1-n}\oplus \bar{h}\oplus h_{n-1}$ 
    induces a symplectic structure follows from  Lemma \ref{l:filtSympConstruct}, one readily checks the required property. 
    One needs to check that $h_{1-n}\oplus \bar{h}\oplus h_{n-1}$ induces a structure of a filtered symplectic vector bundle. 
    This amounts to showing that $\bar{h}$ induces a filtered symplectic structure on $F^{n-2}(\sE)/F^{1-n}(\sE)$. 
    This easily follows from the fact that $h$ induces a filtered symplectic structure on $\sE$. 
    \end{proof}
    
    \begin{Lemma} 
    \label{l:splitFilteredSymplectic}
        Let $(\sE,h)$ be a filtered symplectic vector bundle with  filtration $F$ defined as
         \[    0 = F^{-n}(\sE) \subset F^{-n+1}(\sE)  \subset \hdots \subset F^{n-1}(\sE) =   \sE     \]
        If the sheaves $\gr[i]{F}(\sE)$ are free for all $-n+1 \le i \le n-1$, then $(\sE,h,F)$ is a split, filtered, symplectic 
        vector bundle. 
    \end{Lemma}
        
    \begin{proof}
        This follows by induction using the previous Lemma. 
    \end{proof}

%% file: selfAdjoint.tex
\section{The sheaf of self adjoint homomorphisms}

\subsection{Graded vector bundles.}
Let $\sE=\bigoplus_{i}\sE_{i}$ be a graded vector bundle. We define 
$$
\shom^{\Gr}(\sE,\sE) \subseteq \shom(\sE,\sE)   
$$
to be the subsheaf of homomorphisms that preserve the grading. More generally, if $a$ is an integer, we define 
$$
\shom^{\Gr,\deg a}\subseteq \shom(\sE,\sE) 
$$
to be the subsheaf of degree $a$ homomorphisms. A homomorphism $\phi$ has degree $a$ if the restricted maps satisfy $\phi: \sE_{i}\to \sE_{i+a}$
for every $i$. 
Now suppose that $\sE $ is a graded symplectic vector bundle. So it comes equipped with an isomorphism $h:\sE\to \sE^{\vee}$. Recall that a
homomorphism $f$ is self-adjoint if $\mu(f)=-f$, where $\mu(f)=h^{-1}\circ f^{\vee}\circ h$. By our degree conventions for the dual, if $f$ has 
degree $a$ then so does $f^{\vee}$. We define 
$$
\shom^{\Gr,\deg a}_{s}(\sE,\sE) \subseteq \shom^{\Gr,\deg a}(\sE,\sE) 
$$
to be the subsheaf of self-adjoint homomorphisms. 
 
\begin{Lemma}
    \label{l:explicitDescGraded}
    Consider a graded symplectic vector bundle $(\sE,h)$ on a scheme $S$. Recall that $\sE$ has an explicit description as 
    $$
    \sE = \bigoplus_{i=-n}^{n}\sE_{i}\quad\text{with }h\text{ given by}\quad h_{i}:\sE_{i}\stackrel{\sim }{\to }\sE_{-i}^{\vee}. 
    $$
    Then to give a section of $\shom^{\Gr,\deg a}_{s}(\sE,\sE)$ is the same as giving maps 
    $$
    \theta_{i}:\sE_i\to \sE_{i+a} \quad -n\le i \le n 
    $$
    with $\theta_{-i-a,-i}= -h_{-i}^{-1}\theta^{\vee}_{i,i+a}h_{-i-a}$. In the above, we set $\sE_{i}=0$ if $i\not\in [-n,n]$. 
\end{Lemma}

\begin{proof}
    Follows from \ref{l:gradedSymplectic}. 
\end{proof}

\begin{corollary}\label{c:odd}
    In the above situation, let $a$ be an odd integer. Then we have an isomorphisms 
    \begin{eqnarray*}
        \shom^{\Gr,\deg a}_{s}(\sE,\sE) &\cong& \bigoplus_{i> -a/2} \shom(\sE_{i},\sE_{i+a}) \\ 
      &\cong&   \bigoplus_{i< -a/2} \shom(\sE_{i},\sE_{i+a}). 
    \end{eqnarray*}
\end{corollary}

\begin{Definition}
    \label{d:selfHom}
Let $\sE$ and $\sF$ be vector bundles on a scheme $S$. Let 
$$
h:\sE\to \sF^{\vee}
$$
be an isomorphism. We define $\shom_{s,h}(\sE,\sF)$, the sheaf of self adjoint homomorphisms to be 
the subsheaf
of $\homsheaf(\sE,\sF)$ consisting of local sections $\theta$ such that 
the following diagram commutes
\[ 
\begin{tikzcd}
\sE \arrow[d, "h"] \arrow[r, "\theta"] & \sF \arrow[d, "h^{\vee}"] \\
\sF^{\vee} \arrow[r, "-\theta^{\vee}"]  & \sE^{\vee}               
\end{tikzcd}\]
\end{Definition}

\begin{corollary}\label{c:even}
    In the above situation, let $a=2m$ be an even integer. Then we have an isomorphism 
    \begin{eqnarray*}
        \shom^{\Gr,\deg 2m}_{s}(\sE,\sE) &\cong& \shom_{s,h_{-m}} (\sE_{-m},\sE_{m}) \oplus \bigoplus_{i>m} \shom(\sE_{i},\sE_{i+2m}) \\
        &\cong& \shom_{s,h_{-m}} (\sE_{-m},\sE_{m}) \oplus \bigoplus_{i<m} \shom(\sE_{i},\sE_{i+2m}).    
    \end{eqnarray*}
\end{corollary}

\subsection{Filtered vector bundles.}

\begin{example}
    \label{d:selfAdjointHom}
    In the special case that $\sE=\sF$ in the above and $h$ defines a symplectic vector bundle structure on $\sE$ 
    we recover 
     the \emph{sheaf of self adjoint homomorphisms of $(\sE$,h)}. We will often write
     $$
     \homs(\sE,\sE) := \shom_{s,h}(\sE,\sE). 
     $$
\end{example}

Now suppose that $(\sE,h,F)$ is a filtered symplectic vector bundle with filtration 
\[    0 = F^{-n}(\sE) \subset F^{1-n}(\sE)  \subset \dots \subset F^{n-1}(\sE) =   \sE.   \]
Then $\homs(\sE,\sE)$ comes equipped with a filtration $F_{h}$ of the form 
\begin{align*}
        0 = F_h^{-2n+1}(\homs(\sE,\sE)) &\subset F_h^{-2n+2}(\homs(\sE,\sE))  \subset \hdots\\
                                    &\subset F^{2n-2}(\homs(\sE,\sE)) =  \homs(\sE,\sE).
\end{align*}        
The sheaf $F_h^{j}(\homs(\sE,\sE))$ is defined to be the subsheaf of $\homs(\sE,\sE)$ consisting of elements $f$ 
that satisfy
$$
    f(F^{i}(\sE))\subseteq F^{i+j}(\sE)\quad \forall\ 1-n\le i \le n-1. 
$$

\begin{example}
    Consider a split filtered symplectic vector bundle given by locally free sheaves $\sA_{i}$ for $1-n\le i \le n-1$ and isomorphisms 
    $$
    h_{i}:\sA_{i}\to \sA_{-i}
    $$
    as in Lemma \ref{l:filtSympConstruct}. A self-adjoint homomorphism is given $\theta_{i,j}:A_{i}\to A_{j}$ satisfying the condition 
    of Lemma \ref{l:filtSympConstruct}. Then we have 
    $$
    (\theta_{i,j})\in F_{h}^{\alpha}(\shom(\sE,\sE)) \iff  
    \theta_{i,j}=0 \quad \text{whenever }j>i+\alpha. 
    $$
    This allows us to describe the associated graded sheaf of $\shom_{s}(\oplus \sA_{i},\oplus \sA_{i})$ relatively easily. 
    Write $\sE =\oplus \sA_{i}$ viewed as a filtered symplectic vector bundle. 
\end{example}

\begin{lemma}
    Let $(\sE,h,F)$ be a split filtered symplectic vector bundle. Then we have a canonical isomorphism
    $$
    \Gr^{a}\shom^{\Fil}_{s}(\sE,\sE)\cong \shom^{\Gr,\deg a}_{s}(\Gr \sE,\Gr \sE). 
    $$
\end{lemma}

\begin{proof}
    There is a canonical map 
    $$
    F^{a}\shom^{\Fil}_{s}(\sE,\sE)\cong \shom^{\Gr,\deg a}_{s}(\Gr \sE,\Gr \sE)
    $$
    with kernel $F^{a-1}\shom^{\Fil}_{s}(\sE,\sE)$. So it remains to show that this map is surjective. 
    Consider $\theta$ a section of $\shom^{\Gr,\deg a}_{s}(\Gr \sE,\Gr \sE)$. By Lemma \ref{l:explicitDescGraded}, the map $\theta$ decomposes 
    as $\oplus \theta_{i}$ subject to the given constraint. We define 
    $$
    \theta_{i,j} = \begin{cases}
        \theta^{i} & j=i+a\\
        0 & \text{otherwise}. 
    \end{cases}
    $$
    Then one readily checks that 
    $$
     \bigoplus \theta_{i,j} \in F^{a}\shom^{\Fil}_{s}(\sE,\sE)\quad\text{lifts}\quad \theta. 
    $$
\end{proof}

\begin{corollary}\label{c:gradedReduction}
    Let $(\sE,h,F)$ be a filtered symplectic vector bundle. Then we have a canonical isomorphism
    $$
    \Gr^{a}\shom^{\Fil}_{s}(\sE,\sE)\cong \shom^{\Gr,\deg a}_{s}(\Gr \sE,\Gr \sE). 
    $$
\end{corollary}
 
\begin{proof}
    As in the lemma, there is a canonical map $$
    F^{a}\shom^{\Fil}_{s}(\sE,\sE)\cong \shom^{\Gr,\deg a}_{s}(\Gr \sE,\Gr \sE). 
    $$
    with kernel $F^{a-1}\shom^{\Fil}_{s}(\sE,\sE)$. It remains to show that it is surjective. This question is local,
    so we may assume that $(\sE,h,F)$ is split. This follows from the previous result. 
\end{proof}

\begin{proposition}
    \label{p:degRank}
    Let $(\sE, h, F ) $ be a filtered symplectic vector bundle on a smooth projective curve $X$. 
      Then 
      \begin{enumerate}
          \item $\deg(\Gr^0_F(\homs(\sE,\sE))) = 0$;
          \item $\deg(F^{-2}(\homs(\sE,\sE))) = \deg(F^{1}(\homs(\sE,\sE)))$;
          \item $\rank(\Gr_{F}^i(\homs(\sE,\sE))) =     \rank(\Gr_{F}^{-i}({F}(\homs(\sE,\sE))))$;
          \item $\rank(F^{-i-1}(\homs(\sE,\sE))) + \rank(F^{i}(\homs(\sE,\sE))) = \rank(\homs(\sE,\sE))$.
      \end{enumerate}
\end{proposition}

\begin{proof}
    For all the assertions we will freely apply \ref{c:gradedReduction}. 

    For the first assertion, we remark that it is well-known that $\shom_{s}^{\deg 0}(\sE_{0},\sE_{0})$ has degree 0 as it is 
    the adjoint bundle of the symplectic vector bundle. The assertion follows from \ref{c:even}.

   It follows that 
    $\deg(F^{0}\homs(\sE,\sE)) = \deg(F^{-1}\homs(\sE,\sE))$. Furthermore, using 
    Corollary \ref{c:odd} and the fact that $\Gr_{F}^i(E) = (\Gr_{F}^{-i}(E))^{\vee}$, it follows that  
    \begin{align*}
           (\Gr^{1}_{F}(\shom(\sE,\sE)))^{\vee} &\cong  \bigoplus_{i>-1/2}\shom(\Gr^{i}\sE, \Gr^{i+1}\sE)^{\vee}\\    
           &\cong \bigoplus_{i>-1/2}\shom(\Gr^{i+1}\sE, \Gr^{i}\sE)^{\vee} \\
           &\cong \bigoplus_{i>-1/2}\shom(\Gr^{-i}\sE, \Gr^{-i-1}\sE)^{\vee} \\ 
           &\cong \bigoplus_{i<-1/2}\shom(\Gr^{i}\sE, \Gr^{i-1}\sE)^{\vee}.                 
    \end{align*}
    
    This implies $\deg(\Gr^{-1}_{F}(\homs(\sE,\sE))) = -\deg(\Gr^{1}_{F}(\homs(\sE,\sE)))$. Thus, 
    \begin{align*}
        \deg(F^{1}(\homsheaf_{s}(\sE,\sE))) - \deg(F_h^{-2}(\homsheaf_{s}(E,E))) =  &\deg(\Gr^{-1}_{F}(\homs(\sE,\sE)))\\ 
                                                                         &+ \deg(\Gr^{0}_{F}\homs(\sE,\sE))\\
                                                                         &+ \deg(\Gr^{1}_{F}\homs(\sE,\sE)) = 0,
    \end{align*} 
    and the second assertion is proved.

    For the third claim we write down the calculation for $i$ odd, the even case is similar. We make use of Corollary \ref{c:odd} and the 
    fact that 
    $$
    \rank\Gr^{i}_{F}(\sE) =\rank\Gr^{-i}_{F}(\sE). 
    $$
    
    Then 
    \begin{eqnarray*}
        \rank\Gr^{-i}_{F}(\sE) &=& \sum_{j<i/2}\rank(\Gr^{j}_{F}(\sE))\rank(\Gr^{j-i}_{F}(\sE)) \\
                &=& \sum_{j<i/2} \rank(\Gr^{-j}_{F}(\sE_{j}))\rank(\Gr^{-j+i}_{F}(\sE)) \\
                &=& \sum_{j>i/2} \rank(\Gr^{j}_{F}(\sE_{j}))\rank(\Gr^{j+i}_{F}(\sE)) \\ 
                &=& \rank\Gr^{i}_{F}(\sE). 
    \end{eqnarray*}

    For the last claim we have 
    \begin{eqnarray*}
        \rank(\sE/F^{-i-1}) &=& \sum_{j>=-i-1} \rank \Gr^{j}_{F}(\sE) \\ 
        &=& \sum_{j>-i-1} \rank(\Gr^{-j}(\sE)) \\
        &=& \sum_{j<i+1}\rank(\Gr^{j}(\sE)) \\ 
        &=& \rank (F^{i}). 
    \end{eqnarray*}

\end{proof}

%% file: nilStack.tex
\section{Stacks of nilpotent endomorphisms}

 
\subsection{Filtrations induced by nilpotent endomorphisms.}

\label{s:nilpFilt}

If $\theta$ is a nilpotent endomorphism of a vector bundle $\sE$ we define its \emph{order of nilpotency} to be the 
integer $n$ with 
$$
\theta^{n}=0 \quad\text{and}\quad \theta^{n-1}\ne 0. 
$$

Let $(\sE,h)$ be a symplectic vector bundle over a scheme $S$ with 
 \[ \theta : \sE \rightarrow \sE \]
a self adjoint nilpotent endomorphism of order of nilpotency $n>0$. Then we can equip $(\sE,h)$ with a filtration
 \[    0 = C^{-n}(\sE) \subset C^{1-n}(\sE)  \subset \hdots \subset C^{n-1}(\sE) =   \sE     \]
where 
\[     C^l(\sE) = \underset{\substack{i-j = l \\ 0\le i\le n\\ 1\le j\le n+1}}{\sum} \ker(\theta^i) \cap \image(\theta^{j-1})    \] 
The filtration seems somewhat arbitrary, however it has an important characterization, see \ref{l: charactisation of the filtraion} below.

\begin{Lemma}
\label{Lemma isomorphism from the kth graded piece to -kth graded piece of the canonical filtration}
    With the definitions as above, the map
    \[ \theta^l : C^l(\sE) \rightarrow  C^{-l}(\sE)              \]
    is surjective for $0\le k \le n-1$. Moroever the induced map
    \[  \theta^l : \Gr_{C}^l(\sE) \rightarrow \Gr_{C}^{-l}(\sE)                        \]
    is an isomorphism for $0\le l \le n-1$. 
\end{Lemma}

\begin{proof}
    If $x \in \kerimage{i}{i-l-1}$ then 
    $$
    \theta^{l}(x)\in \kerimage{i-l}{i-1}
    $$
     and hence $\theta^l$
    maps $C^l(E)$ to $C^{-l}(E)$. To show surjectivity, it is enough to show that an element $y\in \kerimage{i}{i+l-1}$ 
    has a preimage $x$. The map 
    \[ \theta^l : \image(\theta^{i-1}) \rightarrow    \image(\theta^{i+l-1})                 \] 
    is surjective so we can certainly find $x\in  \image(\theta^{i-1}) $ with $\theta^{l}(x)=y$.
    Since $y\in \ker(\theta^i)$ we have $x\in \ker(\theta^{i+l})$ thus
    finishing the proof of the first part.  
    Since 
    \[ \theta^l : C^l(\sE) \rightarrow  C^{-l}(\sE)              \]
    is surjective, it's composition with the quotient map $C^l(E)$ $\twoheadrightarrow$ $\frac{C^{-l}(E)}{C^{-l-1}(E)}$
    is surjective as well. 
    
    It remains to show that if $x\in C^l(\sE)$ then $\theta^l(x)\in C^{-l-1}(\sE)$
    if and only if 
    $x\in C^{l-1}(\sE)$. 
    The forward implication has already been proved. 
    
    So suppose that $\theta^{l}(x\in C^{-l-1}(\sE))$. 
    Let $\theta^l(x)$ = $\sum_{i} a_i$ where $a_i$ lies in $\kerimage{i}{i+l}$. 
    By the first part, there
    exists $b_i$ such that $\theta^l(b_i)$ = $a_i$. By a similar argument as above, such a $b_i$ has to lie in $\kerimage{i+l}{i} \subset C^{l-1}(\sE)$. 
    Also $x$ - $\sum_i b_i$ lies in $\ker(\theta^l) \subset C^{l-1}(\sE)$. Hence, $x$ has to lie in 
    $C^{l-1}(\sE)$.
\end{proof}

\begin{Corollary}
\label{Corollary map from graded Ck to graded C(k-2) is injective for k greater than equal to 1}
With definitions as above, the induced map 
\[                 \theta : \Gr_{C}^l(\sE) \rightarrow \Gr_{C}^{l-2}(\sE)          \]
is injective for $n-1\ge l \ge 1$.
\end{Corollary}

\begin{proof}
This follows from the injectivity of $\theta^{l}$. 
\end{proof}

\begin{Lemma}
\label{Lemma k-filtration is perp of -k-1 filtration in Canonical Filtration}
    Let $S$ be a noetherian scheme. 
    Let $(\sE,h)$ be symplectic vector bundle on $S$ and let $\theta$ be a nilpotent 
    endomorphism of $\sE$ of symplectic type with order of nilpotency $n$. Suppose further 
     that $C^i(\sE)/C^{i-1}(\sE)$ is flat over $S$ for all $-n\le i\le n-1$. Then
        \[     C^{l}(\sE)^{\perp} = C^{-l-1}(\sE)                  \]
        for  $-n\le l\le n-1$.
\end{Lemma}
\begin{proof}
We first show that $C^{-1-l}(\sE) \subset C^l(\sE)^{\perp}$. It is enough to show that for  
$x\in \kerimage{i}{i-l-1}$ and $y\in \kerimage{j}{j+l}$ that $h(x)(y)=0$. 

If $y\in  \image(\theta^i) $ then $y=\theta^{i}(y')$. In view of the self adjointness of $\theta$ we have 
$$
h(x)(\theta^{i}(y')) = -h(\theta^{i}(x))(y)=0, 
$$
as $x \in \ker(\theta^{i})$. 

So we may assume that $y\not\in \image(\theta^{i})$. It follows that 
$j+l < i$ and hence 
$j < i-l$. Thus $y\in \ker(\theta^{i-l-1})$ and the result follows from the fact that $x\in \image(\theta^{i-l-1})$.

In view of the noetherian hypothesis on $S$ and the flatness assumptions, all the quotients involving the various $C^{l}(E)$
have a well-defined rank as they are locally free. 
Since
\[\rank(C^l(\sE))   = \sum_{i \leq l} \rank(\Gr_{C}^i(\sE)).          
                  \]

Using Lemma \ref{Lemma isomorphism from the kth graded piece to -kth graded piece of the canonical filtration}
\begin{align*}
    \rank\left(\frac{E}{C^{-1-l}(\sE)}\right) &= \sum_{i > -1-l} \rank(\Gr_{C}^i(\sE))\\
                                &= \sum_{i > -1-l} \rank(\Gr_{C}^{-i}(\sE))\\
                                &= \sum_{i < l+1}   \rank(\Gr_{C}^i(\sE))\\
                                &= \rank(C^l(\sE)).
\end{align*}                                                      
Hence, we have $\rank(C^{-1-l}(\sE)) = \rank(\sE)$ - $\rank(C^l(\sE)) = \rank(C^l(\sE)^{\perp})$. The hypothesis implies that both $\sE/C^{-1-l}(\sE)$ and 
$C^l(\sE)$ are flat. The exact sequence 
\[       0 \rightarrow C^l(\sE)^{\perp} \rightarrow \sE \rightarrow C^l(\sE)^{\vee} \rightarrow 0         \]
implies that $\sE/C^l(\sE)^{\perp}$ is flat as well. Hence, $C^l(\sE)^{\perp}/C^{-1-l }(\sE)$ is flat of $\rank$ 0 and hence 0 itself.
\end{proof}

\begin{Lemma}
    \label{l: charactisation of the filtraion}
        Let $(\sE,h)$ be a symplectic vector bundle over a noetherian scheme $S$ equipped with a nilpotent self adjoint endomorphism $\theta$ and order of nilpotency $n$. 
        Suppose further that
         $\Gr_{C}^i(\sE)$ is flat over $S$ for all $i$. Let $F$ be a filtration on $\sE$ such that 
        \[    0 = F^{-n}(E) \subset F^{1-n}(E)  \subset \hdots \subset F^{n-1}(E) =   E     \]
        and $(\sE,F,h)$ is a filtered symplectic vector bundle. 
        Assume further that 

        \begin{enumerate}
            \item $\theta(F^i(\sE))$ $\subset$ $F^{i-2}(\sE)$
           \item and the induced map 
            \[ \theta^l : \Gr_{F}^l(\sE) \rightarrow \Gr_{F}^{-l}(\sE) \] 
           is injective for all $l$.
          \end{enumerate}
    Then 
        \[     C^{l}(\sE) = F^{l}(\sE)                  \]
        for all $l$.
    \end{Lemma}
    
    \begin{proof}
    It is enough to prove that
        \[      C^l(\sE) = F^l(\sE)                 \]
    for all $n-1\ge l \geq$ $0$ since $C^l(\sE)^{\perp}$ = $C^{-1-l}(\sE)$ by Lemma \ref{Lemma k-filtration is perp of -k-1 filtration in Canonical Filtration} and 
    $F^l(\sE)^{\perp}$ = $F^{-1-l}(\sE)$ by definition. We proceed by descending induction on  $l$. 
    
    For $l=n-1$, we have $C^{l}(\sE) =E =F^l(\sE)$.
    So assume that  $C^{l}(\sE) = F^{l}(\sE)$ for $l \geq i > 0$. 
    We will use the fact $F^{-i-1}(\sE) =F^i(\sE)^{\perp} =C^i(\sE)^{\perp} = C^{-i-1}(\sE)$ repeatedly.

    By definition and Lemma \ref{Lemma isomorphism from the kth graded piece to -kth graded piece of the canonical filtration},
    $\theta^{i}(F^{i-1}(\sE)) \subset F^{-i-1}(\sE) = C^{-i-1}(\sE) = \theta^{i+1}(C^{i+1}(\sE))$.
    Consider $x \in F^{i-1}(\sE)$ then $\theta^{i}(x) = \theta^{i+1}(y)$ for $y \in C^{i+1}(\sE)$. This implies 
    $x \in \ker(\theta^i) + \theta(C^{i+1}(\sE)) \subset C^{i-1}(\sE)$. It follows that $F^{i-1}(\sE)\subseteq C^{i-1}(\sE)$. 
    
    On the other hand $\theta^i(C^{i-1}(\sE)) \subset C^{-i-1}(\sE) = F^{-i-1}(\sE)$. Hence the map 
    \[           \theta^i : F^i(E) \rightarrow   \frac{F^{-i}(E)}{F^{-i-1}(E)}                    \]
    vanishes on $C^{i-1}(E)$. By the injectivity of $\theta^i$ from $\Gr^{i}_{F}(\sE) \to \Gr^{-i}_{F}(\sE)$, we have $C^{i-1}(\sE) \subset F^{i-1}(\sE)$.
    Thus, $C^{i-1}(\sE) = F^{i-1}(\sE)$. This finishes the induction step.
    \end{proof}

 
\subsection{Stacks of nilpotent endomorphisms.}

\label{s:nilpotent}

Let $k$ be a field. Let $X$ be a smooth geometrically connected projective curve over $k$ of genus $g$. 
We introduce several stacks here and prove that they are algebraic. 

  If $\fX$ is a category fibered in groupoids over $\Sch/k$ we will denote by $\fX(S)$ the groupoid of objects over a scheme $S$. 
  We will describe below some categories fibered in groupoids over $\Sch/k$ by describing $\fX(S)$ the entire category fibered in groupoids 
  can be recovered in the following way:

  \begin{enumerate}
    \item in all our constructions there will be evident pullback functors $f^{*}:\fX(S)\to \fX(T)$ for a morphism of schemes $f:T\to S$;
    \item a morphism in $\fX$ over such an $f$ is an isomorphism $f^{*}(x)\stackrel{\sim }{\to}y$ where $x\in \fX(S)$ and $y\in \fX(T)$.
  \end{enumerate}

All stacks below will be stacks in the fppf topology on $\Sch/k$. 

We will make use extensive use of the following theorem from \cite{EGA3}. 

\begin{theorem}
    Let $S$ be a locally noetherian scheme, and $f:Y\to S$ a proper morphism. Let $\sF$ be coherent sheaf on $Y$ that is 
    flat over $S$. Then there is a coherent sheaf $\sI$ on $S$ that induces natural isomorphisms
    $$
    f_{*}(F\otimes_{\sO_S} \sM) \stackrel{\sim }{\longrightarrow} \sheafHom_{\sO_S}(\sI,\sM)
    $$
    and 
    $$
    \Gamma(X, F\otimes_{\sO_S} \sM) \stackrel{\sim }{\longrightarrow} \Hom_{\sO_S}(\sI, \sM ), 
    $$
    where $M$ is a quasi-coherent sheaf on $S$.
    In the first expression (resp. second expression), both sides are viewed as functors 
    $$
    QCoh(S)\to QCoh(S) \quad (resp.\ QCoh(S)\to Ab). 
    $$
    Furthermore, the formation of $\sI $ is compatible with base change. 
\end{theorem}

\begin{proof}
    This is \cite[7.7.6, 7.79]{EGA3}. 
\end{proof}

\begin{Corollary}
    Let $S$ be a locally noetherian algebraic space and $Y\to S$ a proper morphism of algebraic spaces that is representable by schemes. 
    If $\sF$ is a coherent sheaf on $Y$ that is flat over $S$ then there is a coherent sheaf $\sI$ on $S$ with analogous properties. 
\end{Corollary}

\begin{proof}
    There is a locally noetherian scheme $S'$ and etale morphism $S'\to S$. If $\sF'$ is the pullback of $\sF$ to the scheme $Y'=Y\times_{S}S'$ the 
    $\sF'$ come equipped with descent data to $S$. From the theorem there is a sheaf $\sI'$ on $S'$ with the required properties. The descent data 
    for $\sF'$ induces descent data on $\sI'$ and hence there is a sheaf $\sI$ on $S$. One uses descent again to check that $\sI$ has the required properties. 
\end{proof}

\begin{remark}
    \label{r:representability}
    In the above situation, let $\sigma\in \Gamma(Y,\sF)$ be a section. Consider the functor 
    $$
    \Sch/S \to \Sets 
    $$
    given by 
    $$
    (g:T\to S) \mapsto 
    \begin{cases}
        \{*\} & \text{ if }(g\times 1)^{*}(\sigma)=0  \\
        \emptyset & \text{ otherwise}. 
    \end{cases}
    $$
    Then this functor is representable by a closed subspace of $S$. 
     
   To see why, the data of a section is the same as giving a morphism $\sI\to \sO_S$, the image of which is a sheaf of ideals
   which we call $\sJ$. The closed sub-algebraic space defined by $\sJ$ represents this functor in view of the compatibility of the construct
   with base change. 

   A similar argument shows that the functor 
   $$
   \Sch/S \to \Sets 
   $$
   given by 
   $$
   (g:T\to S) \mapsto 
   \begin{cases}
       \{*\} & \text{ if }(g\times 1)^{*}(\sF)=0  \\
       \emptyset & \text{ otherwise}. 
   \end{cases}
   $$
   is also representable. 
\end{remark}

\begin{Definition}
\label{d:BunSp}
    We denote by $\Bun_{X, \Sp_{2r}}$ the moduli stack of $\Sp_{2r}$-principal bundles on $X$ of rank $2r$. 
    The groupoid $\Bun_{X,\Sp_{2r}}(S)$  has objects $(\sE,h)$ that are symplectic vector bundles of rank $2r$ on $X_{S}$. 
    The morphisms in this groupoid are isomorphisms of symplectic vector bundles. 
 \end{Definition}   

 It is well known that $\Bun_{X, \Sp_{2r}}$ is an algebraic stack. 
 
 \begin{Definition}
 \label{Definition stacks of endomorphism of symplectic type}
    We denote by $\Hom_{X, \Sp_{2r}}$ the stack of principal bundles with homomorphism. The groupoid $\Hom_{X, \Sp_2r}(S)$ has 
    as objects triples $(\sE,h,\theta)$ where $(\sE,h)$ is an object of the groupoid $\Bun_{X,\Sp_{2r}}(S)$ and $\theta\in H^{0}(X_{S}, \End(\sE))$
    that is  self-adjoint. 
    A morphism $g:(\sE, h, \theta) \to (\sE', h', \theta')$ is an isomorphism 
              \[      g : (E,h) \rightarrow (E',h')                            \]
         such that $g \theta g^{-1} = \theta'$.
 \end{Definition}

 \begin{proposition}
    The category fibered in groupoids $\Hom_{X,\Sp_{2r}}$ is an algebraic stack. 
 \end{proposition}

 \begin{proof}
    Faithfully flat descent implies that the category $\Hom_{X,\Sp_{2r}}$ is a stack. To complete the proof, it suffices to show that
    the morphism $\Hom_{X,\Sp_{2r}}\to \Bun_{X,\Sp_{2r}}$ is representable. 
    To aid in the proof of representability we introduce a new stack $\widetilde{\Hom}_{X,\Spr}$ whose objects are as in 
    $\Hom_{X,\Spr}$ except that the self-adjointness condition is dropped. We first show that the forgetful morphism 
    $$
    \widetilde{\Hom}_{X,\Spr}\to \Bun_{X,\Spr}
    $$
    is representable. 
    Representability boils down to the following 
    question, given a vector bundle $\sE$ on $X_{S}$ then the functor 
    $$
    \Sch/S \longrightarrow \Sets
    $$
    given by 
    $$
    (f: T\to S) \longmapsto \End (f^{*}(\sE)))
    $$
    is representable. 
    As $\sE$ is locally of finite presentation, we may assume that $S$ is noetherian. Then this follows from 
    \cite[7.7.9]{EGA3} as the $\theta$ is exactly a global section of the locally free sheaf $\sEnd (\sE)$. 

    To complete the proof, we need to show that the faithful functor
    ${\Hom}_{X,\Spr}\hookrightarrow \widetilde{\Hom}_{X,\Spr}$ is representable. 
    This amounts to the condition that $\theta+\mu(\theta)=0$, and we may apply \ref{r:representability}. 
 \end{proof}

 \begin{Definition}\label{d:nilEnd}
 \label{Definition stack of nilpotent endomorphism of symplectic type}
    We denote by $\tNil_{X,\Spr}^n$  the stack of self-adjoint nilpotent endomorphisms of order of nilpotency $n$. 
    The objects of the groupoid $\tNil_{X,\Spr}^{n}(S)$ are triples triple $(\sE,h,\theta)$ as in $\Hom_{X,\Spr}$ with added condition
    that $\theta$ has order of nilpotency $n$. 
 \end{Definition}      

 \begin{proposition}
    The faithful functor
    $$
    \tNil_{X,\Spr}^{n}\to \Hom_{X,\Spr}
    $$
    is representable and hence the stack $\tNil_{X,\Spr}^{n}$ is algebraic. 
 \end{proposition}

 \begin{proof}
    This follows similarly to the previous result. 
 \end{proof}

We refer the reader to \cite[2.4.4, 4.6.2.1]{lmb} for the definitions of $\Coh_{X}$, the stack of coherent sheaves,
and $\Vect_{X,r}$ its open substack of vector bundles. Let's amplify the fact that the objects of $\Coh_{X}$ are locally 
of finite presentation when the base scheme is not noetherian. This allows us to apply a reduction to the noetherian 
case for various arguments. 

\begin{proposition}\label{p:hyperQuot}
    Let $\Coh_{X}^{\Fil_n}$ be the category fibered in groupoids parameterising pairs $(\sE,F)$ where $\sE$ is 
    coherent sheaf on $X_{S}$, locally of finite presentation on $S$ and $F$ is a filtration on $\sE$ of the form 
    $$
    F^{0}=0 \subseteq F^{1}(\sE)\subseteq \ldots \subseteq F^{n}(\sE)=\sE
    $$
    such that all the sheaves $\Gr_{F}^{i}(\sE)$ are flat over $S$. Then the forgetful map 
    $$
    \Coh_{X}^{\Fil_n}\to \Coh_{X}
    $$
    is representable. 
    The substack parameterizing pairs $(\sE,F)$ where $\sE$ is locally free and so are $F^{i}(\sE)$ is an open substack. 
\end{proposition}

\begin{proof}
    The map is representable by quot-algebraic spaces, see \cite[Tag 09TQ]{stacks-project}. 
     The second assertion follows in the same way as showing that the stack of vector bundles 
    is an open substack. 
\end{proof}

\begin{definition}
\label{d:nilStack}
    The stack of filtered self-adjoint nilpotent endomorphisms, denoted $\Nil_{X,\Spr}^{n}$, is the stack whose objects over scheme
    $S$ are quadruples $(\sE,h,\theta, F)$ where $(\sE,h,\theta )$ is an object of $\tNil_{X,\Spr}^{n}$ and $F$ is a filtration on $\sE$
    of the form 
    $$
    0= F^{-n}(\sE) \subseteq \ldots \subseteq F^{n-1}(\sE)=\sE
    $$
    so that $(\sE,h,F)$ becomes a filtered symplectic vector bundle. We require this data to satisfy 
    \begin{enumerate}
        \item $\theta(F^{i}(\sE))\subseteq F^{i-2}(\sE)$ and
        \item the induced map 
        $$
        \theta^{i}:\Gr^{i}_{F}(\sE)\to \Gr^{-i}_{F}(\sE)
        $$
        is injective for all $1-n\le i \le n-1$.
    \end{enumerate}
\end{definition}

\begin{Theorem}
    The stack $\Nil_{X,\Spr}^{n}$ is algebraic. 
\end{Theorem}

\begin{proof}
    As an intermediate step, we first consider the stack $\widehat{\Nil_{X,\Spr}^{n}}$ whose objects are quadruples $(\sE,h,\theta, F)$
    where $(\sE,h,\theta)$ is an object of $\tNil_{X,\Spr}^{n}$ and $F$ is a filtration on $\sE$
    of the form 
    $$
    0= F^{-n}(\sE) \subseteq \ldots \subseteq F^{n-1}(\sE)=\sE.
    $$
    The two remaining conditions defining $\Nil_{X,\Spr}^{n}$ are dropped for  $\widehat{\Nil_{X,\Spr}^{n}}$. The stack  $\widehat{\Nil_{X,\Spr}^{n}}$ is algebraic 
    as the functor 
    $$
     \widehat{\Nil_{X,\Spr}^{n}}\to \tNil_{X,\Spr}^{n}
    $$
    is representable by quot algebraic-spaces as in \ref{p:hyperQuot}. 
    It remains to show that the functor 
    $$
    \Nil_{X,\Spr}^{n}\to \widehat{\Nil_{X,\Spr}^{n}}
    $$
    is representable. 
    The first condition in \ref{d:nilStack} defines a representable morphism using \ref{r:representability} as it corresponds to the vanishing of the 
    morphism 
    $$
    F^{i}(\sE) \to \sE/F^{i-2}(\sE). 
    $$
    Similarly, the second condition defines a representable morphism as it corresponds to the vanishing of kernels. 
\end{proof}

\begin{Theorem}
\label{t:bij on points}
    The natural map 
    $$
    \Nil_{X,\Spr}^{n}\to \tNil_{X,\Spr}^{n}
    $$
    is a bijection on isomorphism classes of $K$-points where $K/k$ is a field extension. 
\end{Theorem}

\begin{proof}
    For injectivity, let us suppose there exists a symplectic vector bundle $(\sE,h)$ on $X_K$ with a
    self-adjoint nilpotent endomorphism $\theta$ with order of nilpotency $n$. Let $(\sE,h,\theta,F_1)$ 
    and $(\sE,h,\theta,F_2)$ be two points of $\Nil_{X,\Spr}^{n}$ mapping to $(\sE,h,\theta)$. By 
    Lemma \ref{l: charactisation of the filtraion}, $F_1$ and $F_2$ agree on an open dense subscheme of 
    the $X_K$. Since $F_1$ and $F_2$ are sections of a flag variety $\Grass(i_1,i_2 \hdots, i_k, \sE)$ which 
    agree on an open set, they are isomorphic.\\ 
    We now have to prove that the morphism is surjective. Again, let us suppose there exists a symplectic vector bundle 
    $(\sE,h)$ on $X_K$ with a self-adjoint nilpotent endomorphism $\theta$ with order of nilpotency $n$. The 
    filtration $C$ (refer Section \ref{s:nilpFilt}) defines a section $s$ of 
    \[       \Grass(i_1,i_2 \hdots i_k, \sE )  \rightarrow X_L                                        \]
    on an open dense subset of $X_L$. By \cite[\href{https://stacks.math.columbia.edu/tag/0BXZ}{Tag 0BXZ}]{stacks-project},
    $s$ can be extended to a section 
    \[ s :  X_K \rightarrow  \Grass(i_1,i_2 \hdots i_k, \sE ).                  \]
    This defines a filtration $F$ on $\sE$ such that 
  \[   \theta(F^i(\sE)) \subset F^{i-2}(\sE) \] 
    and
\[    \theta^i :  \Gr^i_F(\sE) \rightarrow \Gr^{-i}_F(\sE)                              \]
    is injective on an open dense subset of $X_K$. The first condition is a closed condition and hence satisfied on the 
    whole curve $X_K$. The second condition implies that the kernel of  
    \[    \theta^i :  \Gr^i_F(\sE) \rightarrow \Gr^{-i}_F(\sE)                              \]
    is 0 on an open dense subset of $X_K$, it is 0 on $X_K$. Thus, $(\sE,h,\theta,F)$ is $K$-point of 
    $ \Nil_{X,\Spr}^{n}$ that maps to $(\sE,h,\theta)$.
\end{proof}

 
\subsection{Deformations and obstructions.}

We assume in this subsection that $k$ is an algebraically closed field.

\begin{Definition}
Suppose that $\fX$ is an algebraic stack over $\Sch_{k}$. 
Given an epimorphism of local artinian rings $B \rightarrow A$ and an object $x\in\fX_{A}$ a \emph{lift of $x$ to $B$} is an object $y$ in $\fX_{B}$ and a 
morphism $f:x\to y$ in $\fX$.  
\end{Definition}
    
\begin{Example}
    Consider the stack $\Bun_{X,\Spr}$ and a point $(\sE,h)$ over $\Spec A$. Then a lift of this point to $\Spec B$ is a symplectic vector bundle $(\sE',h')$
    over $X_{B}$ of rank $2r$ together with an isomorphism $\sE'\otimes_{B} A\cong \sE$ inducing an isomorphism of symplectic vector bundles. 
\end{Example}

\begin{Lemma}
    \label{l:injectiveLift}
    Let 
    $$
    0 \rightarrow I \xrightarrow{i} B \xrightarrow{q} A \rightarrow 0.  
    $$
    be a square zero extension of local artinian $k$-algebras. 
    Let $R$ be an integral domain and let $\theta$ be an injective map of free $R \otimes_l A$ modules  
    \[   \theta : (R \otimes_l A)^k \rightarrow (R \otimes_l A)^n                          \]
   and let 
    \[  \psi:  (R \otimes_l B)^k \rightarrow (R \otimes_l B)^n  \]
   lift $\phi$. Then $\psi$ is injective.
\end{Lemma}
    
\begin{proof}
     Since any map of free modules over a ring $S$ 
     \[       f : S^k \rightarrow S^n                               \] 
     is injective if and only if the map 
     \[       \bigwedge^k f : \bigwedge^k S^k \rightarrow \bigwedge^k S^n                                         \]
    is, we can reduce to the case $k = 1$. In this case, we prove by contradiction. We suppose that the map 
    \[      \psi : R \otimes B \rightarrow (R \otimes B)^n                                                    \]
    is not injective. We choose a basis $\{ a_1, a_2, \hdots a_n\}$ of $(R \otimes_l A)^n$
    and a basis $\{b_1, b_2, \hdots, b_n\}$ of $(R \otimes_l B)^n$ such that the epimorphism 
    \[         (R \otimes B)^n \twoheadrightarrow (R \otimes A)^n                                          \]
    sends $b_i$ to $a_i$.   Expressing $\psi$ in this basis of $(R \otimes B)^n$, we get 
    \[      \psi(1) = \sum_{i=1}^n \psi_i b_i,                                 \]
    with $\psi_i$ in $R \otimes_l B$. Non injectivity of $\psi$  implies that $\psi_i$ are zero divisors for all $i$. Since $R = R \otimes_l l$
    has an unique associated prime ideal, $R \otimes B$ has an unique associated prime ideal as well (\cite[Theorem 1]{LQ}).
    This implies that $\psi_i$ are nilpotent for all $i$. On the other hand, since $\psi$ lift $\phi$,
     \[      \phi(1) = \sum_{i=1}^n \phi_i a_i                               \]
    with  $\phi_i = (id_R \otimes q) (\psi_i)$, where $id_R \otimes q$ is the map in the exact sequence      
    \[   0 \rightarrow R \otimes_l I \xrightarrow{id_R \otimes i} R \otimes_l B \xrightarrow{id_R \otimes q} R \otimes_l A \rightarrow 0 . \]
     Hence, $\psi_i$ 
    being nilpotent implies that $\phi_i$ is nilpotent for all $1 \le i \le n$ and thus belong to the unique 
    associated prime of $R \otimes_l A$. Hence, there exists an element $r$ of $R\otimes_l A$ that annihilates
    $\phi_i$ for all $1\le i \le n$ contradicting the injectivity of $\phi$.   
\end{proof}

\begin{Lemma}
    \label{l:symplecticLift}
       Let $A$ be an artinian local $k$-algebra and  
        Let $U = \Spec(R)$ be an affine open subscheme of $X$. Suppose that $(\sE,h)$ is a symplectic vector bundle on $U_{A}$ such 
        that $\sE$ is a free $R_{A}$-module of rank $2r$. If 
        $$
        0\to I\to B\to A\to 0
        $$
        is a square zero extension of artinian local $k$-algebras
        then there exists a free symplectic vector bundle $(\sE',h')$ on $R \otimes_l B$ that lifts $(\sE,h)$. 
\end{Lemma}

\begin{proof}
    We can lift the free module $E$ to a free module $E'$. We can also lift the isomorphism 
    $h : \sE \rightarrow \sE^{\vee}$ 
    to an isomorphism $ \tilde{h} : \sE' \rightarrow (\sE')^{\vee}$.
    We define
    \[           h' : \sE' \rightarrow (\sE')^{\vee} \quad\text{as}\quad     h' = \frac{\tilde{h} - \tilde{h}^{\vee} \circ \ev }{2}.                            \]
     This map lifts $\frac{h - h^{\vee} \circ \ev}{2} = h$. Since
    $h$ is surjective, we have that $h'$ is surjective and thus isomorphism. Moreover, since 
    \[    (h^{\vee} \circ \ev )^{\vee} \circ \ev  = h,                                    \]    
    we can compute that 
    \begin{align*}                                             
        h'^{\vee} \circ \ev      &= \frac{h^{\vee} \circ \ev  -  (h^{\vee} \circ \ev )^{\vee} \circ \ev }{2} \\ 
                                    &= \frac{h^{\vee} \circ \ev  - h}{2}\\ 
                                    &= -h'. 
    \end{align*}
    Hence $(\sE',h')$ is a symplectic vector bundle. 
\end{proof}

\begin{Lemma}
\label{l:filteredSymplecticLift}
    Let $U = \Spec(R)$ be an affine open subscheme of $X$ and let $A$ be an artinian local $k$-algebra. 
    Let $(\sE,F,h)$ be a filtered symplectic vector bundle on $U_{A}$
    such that the associated graded sheaf $\Gr_F^i(\sE)$ is free on 
    $R \otimes_l A$ for $-n+1 \le i \le n-1$. If 
    $$
    0\to I\to B\to A\to 0
    $$
    is a square zero extension of artinian local $k$-algebras then there exists a filtered symplectic vector bundle $(\sE',F',h)$ on $R \otimes_k B$ 
    that lifts $(\sE,F, h)$. 
\end{Lemma}

\begin{proof}
    By Lemma \ref{l:splitFilteredSymplectic} $(\sE,F,h)$ is a split filtered symplectic vector bundle. 
    As such, it is isomorphic to its associated graded filtered symplectic vector bundle. The symplectic structure can thus be 
    described by isomorphisms 
    \[       h_i : \Gr_F^i(\sE) \rightarrow \Gr_F^{-i}(\sE)^{\vee}                                  \]
    for $-n+1 \le i \le n-1$, satisfying $h_i = - h_{-i}^{\vee} \circ \ev $. Since $\Gr_F^i(\sE)$ is free for all $i$, 
    we choose free $R\otimes_{k}B$-modules $\sE_{i}'$ that lift $\Gr_F^i(E)$ for $-n+1 \le i \le n-1$. One chooses isomorphisms
    \[     h'_i : \sE'_{i} \rightarrow \sE_{-i}'                   \] 
    for $i>0$ that lift $h_i$. Using Lemma \ref{l:symplecticLift}, we can define isomorphism 
    \[    h'_0 : \sE_{0}'\to \sE_{0}'                        \]  
    that lifts $h_0$ and satisfies $h'_0 = - (h'_0)^{\vee} \circ \ev$. For $i < 0$, we define 
    \[     h'_i : \sE_{i}\to \sE_{-i}\quad\text{by}\quad      h'_i = - (h'_{-i})^{\vee} \circ \ev.                   \]
    By construction, $\sE':=\overset{n-1}{\underset{i= -n+1}{\oplus}}   \sE_{i}$ lifts $\overset{n-1}{\underset{i= -n+1}{\oplus}}  \Gr_{F'}^i(\sE)$ 
    and the modules 
    \[        (F')^k(\sE')= \overset{k}{\underset{i= -n+1}{\oplus}} \sE_{i}'           \]
    lifts $F^k(E)$ for $-n+1 \le k \le n-1$. One applies \ref{l:filtSympConstruct}
    to see that the lifted filtered symplectic vector bundle has a filtered symplectic structure. 
\end{proof}

\begin{Lemma}
    \label{l:existenceLiftsNilp}
Let $A$ be an artinian local $k$-algebra. Consider an $A$-point $(\sE,h,\theta,F)$ of $\Nil^{n}_{X,\Spr}$
Let $U = \Spec(R)$ be an affine open subscheme of $X$ such that $\Gr_F^i(\sE)$ is free for all $-n+1 \le i \le n-1$.
If 
$$
0\to I\to B\to A\to 0
$$
is a square zero extension of artinian local $k$-algebras then
there exists a quadruple $(\sE',h',\theta',F')$ over $U\otimes_{A}B$ lifting $(\sE|_{U}, h,\theta|_{U},F|_{U})$. 
The lift can be chosen so that 
\begin{enumerate}
    \item $\theta'((F')^{i}(\sE'))\subseteq (F')^{i-2}(\sE')$ and
    \item the induced map 
    $$
    \theta^{k}:\Gr_F^i(\sE)\to \Gr_F^{-i}(\sE)
    $$
    is injective for all $-n+1 \le i \le n-1.$ 
\end{enumerate}
\end{Lemma}

\begin{proof}
    By the previous Lemma we can certainly obtain the lift $(\sE',h',F')$. Furthermore, as $\Gr^{i}_{F'}(\sE)$
    is a flat deformation of the free module $\Gr^{i}_{F}(\sE)|_{U}$, is itself free. If follows from Lemma \ref{l:splitFilteredSymplectic}
    that $(\sE',h',F')$ is split filtered symplectic vector bundle. It follows that $h'$ is determined by the isomorphisms
    $$
    h'_{i}:\Gr^{i}_{F'}(\sE) \to \Gr^{-i}_{F'}(\sE)^{\vee}. 
    $$
    The final step is to lift the self adjoint $\theta$
     to a self adjoint morphism $\theta'$ that satisfies the properties (1) and (2) above. 
    The map $\theta$ can be written as 
$\oplus_{-n+1 \le i,j \le n-1} \theta_{i,j}$ on $\Gr_F(E)$
with 
\begin{equation}
\theta_{i,j} =  - h_j^{-1}  \theta_{-j,-i}^{\vee} h_i                                
\end{equation}
by Lemma \ref{l:filtSympConstruct}. Moreover, $\theta_{i,j} = 0$ for $i-j < 2$
since $\theta(F^{i}(\sE)) \subset F^{i-2}(\sE)$. 
As the vector bundles $\Gr^{i}_{F}(\sE')$ are free we can choose 
$$
A_{i,j}:\Gr^{i}(\sE')\to \Gr^{j}(\sE')
$$
lifting $\theta_{i,j}$. 
We define 
\[     \theta'_{i,j} : \Gr^i_{F'}(E') \rightarrow \Gr^{j}_{F'}(E')                            \]
by
\[ 
\theta'_{i,j} = \begin{cases}
    0       & \text{when}\ i-j<2\\ 
    A_{i,j} & \text{when}\ i-j\ge 2\ \text{and}\ i>|j| \\
    -(h'_{j})^{-1}(A_{-j-i})^{\vee}h_{i} & \text{when}\ i-j\ge 2\ \text{and}\ i<|j|. \\
    \frac{A_{i,-i}-  (h'_{-i})^{-1}A_{i,-i}^{\vee}h_{i}'}{2}                  & \text{when }i=-j\text{ and }i>0. 
\end{cases}
\]


We define $\theta'= \oplus_{i,j} \theta'_{i,j}$. This lifts $\theta$ since $\theta'_{i,j}$ lifts $\theta_{i,j}$ for all 
$i$,$j$. As $\theta_{i,j}=0$ when $i-j<2$ it follows that 
$$
\theta' ((F')^{i}(\sE') \subseteq (F')^{i-2}(\sE'))
$$
The lift is self-adjoint by Lemma \ref{l:filtSympConstruct}. 
Lastly, the induced maps 
\[   (\theta')^i : \Gr_{F'}^i(\sE') \rightarrow \Gr^{-i}_{F'}(\sE')                                                       \]
are injective since they lift the  injective maps (refer Lemma \ref{l:injectiveLift}) 
\[   \theta^i : \Gr^i_{F}(\sE') \rightarrow \Gr^{-i}_{F}(\sE').                                                        \]
\end{proof}

\begin{Lemma}
\label{l:uniquenessLift}
Let $\Spec(R)=U\subseteq X$ be an open affine subset. 
Consider a filtered symplectic vector bundle $(\sE_{0},F_{0},h_{0})$ on $U$ 
 such that $\Gr^i_{F_0}(\sE_0)$ is free for all $-n+1 \le i \le n-1$.
Let $(\sE_1,h_1, F_1)$ and $(\sE_2,h_{2}, F_2)$ be two filtered symplectic vector bundles over $R\otimes_l B$ that lift  the filtered 
symplectic vector bundle $(\sE_{0},h_{0},F)$. Suppose that both the associated graded sheaves $\Gr_{F_1}(\sE_1)$ and $\Gr_{F_2}(\sE_2)$ are flat
over $B$. 
Then there exists an isomorphism of  filtered symplectic vector bundles
\[       \sigma : \sE_1 \rightarrow \sE_2                              \]
that lifts the identity map on $\sE_{0}$. 
\end{Lemma}

\begin{proof}
   The modules $\Gr^i_{F_1}(\sE_1)$ and $\Gr^i_{F_2}(\sE_2)$  for $-n+1 \le i \le n-1$ are free modules
 since they are flat lifts of free modules. By \ref{l:splitFilteredSymplectic} the filtered 
 symplectic vector bundles $\sE_i$ are in fact split filtered symplectic vector bundles. So there is a decomposition 
 $$
 h_{\alpha} = \bigoplus_{i} h_{\alpha, i} \quad \text{where}\quad
 h_{\alpha,i}:\Gr^{i}_{F_{\alpha}}(\sE_{\alpha})\to \Gr^{-i}_{F_{\alpha}}(\sE_{\alpha}),
 $$
 as in \ref{l:filtSympConstruct} for $\alpha = 1,2$. 
 As any two symplectic forms on a free module are isomorphic, there is an isomorphism 
 $\Gr^{0}_{F_1}(\sE_{2})\cong \Gr^{0}_{F_2}(\sE_{2}) $ commuting with the $h_{\alpha,0}$'s. 
 Similarly, one can choose isomorphisms $\Gr^{i}_{F_1}(\sE_{2})\cong \Gr^{i}_{F_2}(\sE_{2})$ for $i>0$
 so that the following diagrams commute:

 \begin{center}
    \begin{tikzcd}
        \Gr^{i}_{F_1}(\sE_{2}) \ar[r]\ar[d,"h_{1,i}"] & \Gr^{i}_{F_2}(\sE_{2}) \ar[d, "h_{2,i}"] \\
        \Gr^{-i}_{F_1}(\sE_{2}) \ar[r] & \Gr^{i}_{F_2}(\sE_{2}).
    \end{tikzcd}
 \end{center}
 There are induced isomorphism for $i<0$ that produce an isomorphism of filtered symplectic vector bundles. 

\end{proof}

\subsection{The obstruction and deformation space.}

We assume in this subsection that $k$ is an algebraically closed field.

 Given a $k$-point
 $(\sE_0,h_{0},\theta_0,F_0)$ of $\Nil_{X,\Spr}^{n}$ 
 we define a two term complex
 \[            P(\sE_0,h_{0},\theta_0,F_0) = (F_0)_h^0(\homs(E_0,E_0)) \xrightarrow{[\theta_0,-]} (F_0)_h^{-2}(\homs(E_0,E_0)).      \]
 Where $[\theta_{0},f]=\theta_{0}f - f\theta_{0}$. 

 \begin{Theorem}
\label{t:deformations}
    Let 
    \[      0 \to I \to B \to A \to 0                           \]
    be a square zero extension of local artinian $k$-algebras. Let  $(\sE_0,h_0,\theta_0,F_0)$ be a $k$-point of 
    $\Nil_{X,\Spr}^{n}$. 
    If $(\sE,h,\theta, F)$ is an $A$-point of $\Nil_{X,\Spr}^{n}$ lifting $(\sE_0,h_0,\theta_0,F_0)$ then 
    \begin{enumerate}
        \item there exists an obstruction to lifting $(\sE,h,\theta,F)$ to a $B$-point in 
        $$\Cohomology^2(X, P(\sE_0,h_{0},\theta_0,F_0) \otimes_{k} I);$$
        \item if the obstruction vanishes, the space of lifts is a torsor under 
        $$\Cohomology^1(X, P(\sE_0,h_{0},\theta_0,F_0) \otimes_{k} I);$$
        \item the automorphism group of a lift is isomorphic to 
        $$\Cohomology^0(X, P(\sE_0,h_{0},\theta_0,F_0) \otimes_{k} I).$$
    \end{enumerate}
 \end{Theorem}

 \begin{proof} In what follows we write $P:=  P(\sE_0,h_{0},\theta_0,F_0)$.
    Using Lemma 
    \ref{l:existenceLiftsNilp} we can find an open affine cover $\{U_{\alpha}\}_{\alpha \in I}$ of $X$
    so that our given $A$-point lifts locally to $B$-points  $(\sE_{\alpha},h_{\alpha},\theta_{\alpha},F_{\alpha})$. 
    
    Lemma \ref{l:uniquenessLift} shows that we have isomorphisms of filtered symplectic vector bundles 
   \[ g_{\alpha\beta} : \sE_{\alpha}|_{U_{\alpha} \cap U_{\beta}} \rightarrow \sE_{\beta}|_{U_{\alpha} \cap U_{\beta}}  \]
   on the double intersections $U_{\alpha} \cap U_{\beta}$. 
    On triple intersections $U_{\alpha} \cap U_{\beta} \cap U_{\gamma}$, 
   we define $c_{\alpha\beta\gamma}$ $\in$ $\homsFz(E_0,E_0 \otimes I)|_{U_{\alpha} \cap U_{\beta} \cap U_{\gamma}}$
    \[          c_{\alpha\beta\gamma} = 1 - g_{\alpha\gamma}^{-1} g_{\beta\gamma}g_{\alpha\beta}                    \]
and $\tau_{\alpha\beta}$ $\in$ $\homsFz(E_0,E_0(2) \otimes I)|_{U_{\alpha} \cap U_{\beta}}$
    \[           \tau_{\alpha\beta} = \theta_{\alpha} - g_{\alpha\beta}^{-1}\theta_{\beta}g_{\alpha\beta}.                                                                                                 \]
It is easy to check that the pair $(c_{\alpha\beta\gamma},a_{\alpha\beta})$ is a 2-cocycle for $P \otimes I$ and 
the vanishing of the corresponding cohomology class implies the existence of a global lift. 

For the second part, if there exists two lifts $(\sE_1,h_{1},\theta_1,F_1)$ and $(\sE_2,h_{2},\theta_2,F_2)$, then by 
\ref{l:uniquenessLift}, there exists an isomorphism of filtered symplectic vector bundles 
\[    \phi_{\alpha} : \sE_1|_{U_{\alpha}} \rightarrow  \sE_2|_{U_{\alpha}}.                                     \]
 Hence we define a 1-cocycle for $P \otimes I$,  
$(c_{\alpha\beta}, \tau_{\alpha})$ 
\[       c_{\alpha\beta} = \phi_{\beta}^{-1} g_{\alpha\beta} \phi_{\alpha} - 1                                   \]
\[       \tau_{\alpha} = \phi_{\alpha}^{-1}\theta_2\phi_{\alpha} - \theta_1                                          \]
where $g_{\alpha\beta}$ is a transition map
\[ g_{\alpha\beta} : \sE_{2}|_{U_{\alpha} \cap U_{\beta}} \rightarrow \sE_{2}|_{U_{\alpha} \cap U_{\beta}}  \]
The corresponding cohomology class vanishes if and only if the lifts are isomorphic. 

For the third part,  $\phi$ is an automorphism if and only if $1 - \phi$ is a global section of $P \otimes I$.
\end{proof}

\subsection{The dimension and smoothness $\Nil_{X,\Spr}^{n}$}
\label{s:dimension}

We will assume in this section that $k$ is an algebraically closed field of characteristic 0.

\subsection{Vanishing of the obstruction.}

\begin{Lemma}
\label{l:graded surjective implies map surjective}
    Let $(\sE_1,F_1)$ and $(\sE_2,F_2)$ be filtered vector bundles over a scheme $X$. 
    Let $\theta$ be a morphism of filtered vector bundles $(\sE_1,F_1)$ and $(\sE_2,F_2)$. Then $\theta$ is surjective 
    if $\Gr \theta$ is surjective.
\end{Lemma} 

\begin{proof}
We prove by induction on $i$ that 
\[    \theta|_{F_1^i(\sE_1)} : F_1^i(\sE_1) \rightarrow F_2^i(\sE_2)                       \]
is surjective for all $i$. Let $m$ be such that $F_1^m(\sE_1) = F_2^m(\sE_2) = 0$. Then clearly 
\[    \theta|_{F_1^m(\sE_1)} : F_1^m(\sE_1) \rightarrow F_2^m(\sE_2)                                 \]
is surjective. We assume that 
\[       \theta|_{F_1^i(\sE_1)} : F_1^i(\sE_1) \rightarrow F_2^i(\sE_2)                                   \] 
is surjective. Then we have the following commutative diagram 
\[
\begin{tikzcd}
0 \arrow[r] & F_1^i(\sE_1) \arrow[r] \arrow[d, "\theta|_{F_1^{k}(\sE_1)}"] & F_1^{i +1}(\sE_1) \arrow[r] \arrow[d, "\theta|_{F_1^{k+1}(\sE_1)}"] & \Gr^{i +1}_{F_1}(\sE_1) \arrow[r] \arrow[d, "\Gr^{k+1}\theta"] & 0 \\
0 \arrow[r] & F_2^i(\sE_2) \arrow[r]                                       & F_2^{i +1}(\sE_2) \arrow[r]                                         & \Gr^{i +1}_{F_2}(\sE_2) \arrow[r]                               & 0
\end{tikzcd}\]
The first and the third vertical arrows are surjective by the induction hypothesis and the assumption on $\theta$ respectively. 
Hence, $\theta|_{F^{k+1(E)}}$ is surjective by the five lemma.
\end{proof}

\begin{Lemma}
\label{l:[theta,-] and gr commutes}
Let $(\sE,h,F)$ be a filtered symplectic vector bundle. Let $\theta$ be a self-adjoint endomorphism of $\sE$ such that 
$\theta(F^a(\sE)) \subset F^{a-2}(\sE)$. Then for all $a$, the following diagram is commutative 
\[
\begin{tikzcd}
{\Gr^{a}\shom^{\Fil}_{s}(\sE,\sE)} \arrow[rr, "\sim"] \arrow[d, "{\Gr^a[\theta,-]}"] &  & {\shom^{\Gr,\deg a}_{s}(\Gr \sE,\Gr \sE)} \arrow[d, "{[\Gr\theta,-]}"] \\
{\Gr^{a-2}\shom^{\Fil}_{s}(\sE,\sE)} \arrow[rr, "\sim"]                              &  & {\shom^{\Gr,\deg a-2}_{s}(\Gr \sE,\Gr \sE)}                           
\end{tikzcd}
\]
where the horizontal arrows are in Corollary \ref{c:gradedReduction}
\end{Lemma}

\begin{proof}
    It is enough to check that the following diagram 
    \[
\begin{tikzcd}
{F^{a}\shom^{\Fil}_{s}(\sE,\sE)} \arrow[rr, two heads] \arrow[d, "{\Gr^a[\theta,-]}"] &  & {\shom^{\Gr,\deg a}_{s}(\Gr \sE,\Gr \sE)} \arrow[d, "{[\Gr\theta,-]}"] \\
{F^{a-2}\shom^{\Fil}_{s}(\sE,\sE)} \arrow[rr, two heads]                              &  & {\shom^{\Gr,\deg a-2}_{s}(\Gr \sE,\Gr \sE)}                           
\end{tikzcd}
\]
The horizontal maps are 
\[     \alpha \mapsto \Gr \alpha.              \]
Hence, this follows from 
\[ \Gr [\theta, \alpha] = \Gr(\theta \alpha - \alpha \theta) = \Gr \theta \Gr \alpha - \Gr \alpha \Gr \theta = [\Gr \theta, \Gr \alpha ].  \]
\end{proof}

\begin{Lemma}
    \label{l:surjectivity on associated graded vector space case}
 Let $(V,h,F)$ be a filtered symplectic vector space. Let $\theta$ be a self-adjoint endomorphism of $V$ such that 
$\theta(F^a(V)) \subset F^{a-2}(V)$. Then,
\[   [\theta,-] : F_h^i(\Hom_s(V,V)) \rightarrow F_h^{i-2}(\Hom_s(V,V))          \]
is surjective for all $i \le 1$.
\end{Lemma}

\begin{proof}
   By Lemma \ref{l:graded surjective implies map surjective} and Lemma \ref{l:[theta,-] and gr commutes}, it is enough to prove that 
   \[    [\Gr\theta,-] : \Hom_s^{\Gr,\deg i}(\Gr V,\Gr V) \rightarrow \Hom_s^{\Gr,\deg i-2}(\Gr V,\Gr V)                         \] 
    for all $i \le 1$. Note that using the isomorphism 
    \[           \Gr^i(\Hom_s^{\Fil}(V,V)) =  \Hom_s^{\Gr,\deg i}(\Gr V,\Gr V),           \]
    this is part of the definition of a good grading on $\Hom_s(V,V)$ (refer \cite[Equation 0.2]{VK}) with the negative convention. To prove, this defines 
    a good grading, we define an element $h(\theta)$ in $\Hom_s^{\Gr,\deg 0}(\Gr V,\Gr V)$ by the equation 
    \begin{align*}
        \alpha = \oplus_i \alpha_i ,   \quad \quad \alpha_i : \Gr^i V &\rightarrow \Gr^i V  \\  
                                                    v  &\mapsto i.v. 
    \end{align*}
    This is of the adjoint type since 
    \[
    \begin{tikzcd}
    \Gr^i(V) \arrow[d, "h_i"] \arrow[r, "\alpha_i"] & \Gr^i(V) \arrow[d, "h_i"] \\
    \Gr^{-i}(V)^\vee \arrow[r, "-\alpha_{-i}^\vee"] & \Gr^{-i}(V)^\vee .       
    \end{tikzcd}\]
    Furthermore, it is clear by construction that the map 
    \[      [\alpha,-] :   \Hom_s^{\Gr,\deg i}(\Gr V,\Gr V) \rightarrow \Hom_s^{\Gr,\deg i }(\Gr V,\Gr V)                         \]
    is multiplication by $i$. In particular $[\alpha, \Gr \theta] = -2[\alpha, \Gr \theta]$. By \cite[Theorem 2.12, Claim 2, Page 26]{Ali}
    and \cite[Lemma 2.10]{Ali}, this pair $(\Gr \theta, \alpha)$ can be extended to a Jacobson Morozov triple 
    $(\Gr \theta, \alpha, f)$ such that 
    \[        [\alpha,f] = 2f, \quad [\Gr\theta, f] = \alpha \quad and \quad [\alpha,\Gr \theta] = -2\Gr \theta      \]
    and clearly the grading on $\Hom_s(V,V)$ is the corresponding Dynkin grading. This grading is now good by 
    \cite[Proposition 3.11]{Ali}. 
\end{proof}

\begin{Lemma}
    \label{Lemma cokernel is torsion}
    Let $X$ be a smooth projective curve. Let $(\sE,F, h, \theta )$ be a filtered symplectic vector bundle on $X$ with a self-adjoint 
    nilpotent endomorphism. 
    Then 
the map 
\[    [\theta,-] :    (F_0)_h^0(\homs(\sE,\sE)) \rightarrow (F_0)_h^{-2}(\homs(\sE,\sE))          \]
has torsion cokernel.
\end{Lemma}

\begin{proof}
    It is enough to prove that the map
    \[    [\theta,-] :    (F_0)_h^0(\homs(\sE,\sE)) \rightarrow  (F_0)_h^{-2}(\homs(\sE , \sE))          \]
    is surjective after passing to the generic point. Thus it is enough to prove it for a filtered symplectic vector space 
    $(V,F, h)$ and a nilpotent self adjoint endomorphism $\theta$ with 
    order of nilpotency $n$. This now follows from Lemma \ref{l:surjectivity on associated graded vector space case}.
\end{proof}

\begin{Theorem}
\label{Theorem Vanishing of the Obstruction}
   Let $X$ be smooth projective curve over $k$.  
   Let $(\sE_{0},h_{0},\theta_{0},F_{0})$ be a $k$-point of $\Nil^{n}_{X,\Spr}$. 
   If $P=P(\sE_{0},h_{0},\theta_{0},F_{0})$ then 
    the second cohomology $\Cohomology^2(X,P)$ vanishes and hence $\Nil_{X,\Spr}^{n}$ is smooth.
\end{Theorem}

\begin{proof}
We have a long exact sequence 
\[       \Cohomology^1(X, (F_0)_h^0(\homs(\sE_0,\sE_0))) \xrightarrow{[\theta_0,-]} \Cohomology^1(X,(F_0)_h^{-2}(\homs(\sE_0,\sE_0))) \rightarrow \Cohomology^2(P) \rightarrow 0 \] 
since $\Cohomology^2(X, (F_0)_h^0(\homs(\sE_0,\sE_0)))$ = 0. Thus it is enough to prove that the map 
\[    [\theta_0,-] :    \Cohomology^1(X, (F_0)_h^0(\homs(\sE_0,\sE_0))) \rightarrow  \Cohomology^1(X,(F_0)_h^0(\homs(\sE_0,\sE_0)))           \]
is surjective. 

By Lemma \ref{Lemma cokernel is torsion}, the cokernel of the map
\[    [\theta_0,-] :   (F_0)_h^0(\homs(\sE_0,\sE_0)) \rightarrow  (F_0)_h^{-2}(\homs(\sE_0,\sE_0))          \]
is torsion. Hence, $\Cohomology^1(\coker([\theta_0,-]))$ = 0. \
\end{proof}  

\subsection{The dimension calculation.}

In this subsection we let 
 $X$ be smooth projective curve over $k$.  
 We will further fix $(\sE_{0},h_{0},\theta_{0},F_{0})$ to be a $k$-point of $\Nil^{n}_{X,\Spr}$.

\begin{Lemma}
\label{Lemma cokernel of map of associated graded is torsion}
The map 
\[      [\theta_0,-] :   \Gr^1_{(F_0)_h}(\homs(\sE_0,\sE_0)) \rightarrow \Gr^{-1}_{(F_0)_h}(\homs(\sE_0,\sE_0))                             \]
has torsion cokernel.
\end{Lemma}

\begin{proof}
Lemma \ref{l:surjectivity on associated graded vector space case} proves that the map
\[     [\theta_0,-] :    (F_0)_h^{1}(\homs(\sE_0,\sE_0)) \rightarrow (F_0)_h^{-1}(\homs(\sE_0,\sE_0))                      \]
is surjective for $E_0$ a vector space. This induces a surjective map 
\[      [\theta_0,-] :   \Gr^1_{(F_0)_h}(\homs(\sE_0,\sE_0)) \rightarrow \Gr^{-1}_{(F_0)_h}(\homs(\sE_0,\sE_0))                             \]
for $E_0$ a vector space. The result follows.
\end{proof}

\begin{Lemma}
\label{Lemma degrees is upped bounded}
The degrees of $(F_0)_h^0(\homs(\sE_0,\sE_0))$ and $(F_0)_h^{-2}(\homs(\sE_0,\sE_0))$ satisfy the following 
\[ \deg((F_0)_h^0(\homs(\sE_0,\sE_0))) - \deg((F_0)_h^{-2}(\homs(\sE_0,\sE_0))) \ge 0. \]                                                                             
\end{Lemma}
\begin{proof}
The map 
\[      [\theta_0,-] :   \Gr^{1}_{(F_0)_h}(\homs(\sE_0,\sE_0)) \rightarrow \Gr^{-1}_{(F_0)_h}(\homs(\sE_0,\sE_0))                             \]
is surjective on an open dense set by Lemma \ref{Lemma cokernel of map of associated graded is torsion} and both 
sides have the same rank by Lemma \ref{p:degRank} part (3). Hence, the kernel of this map vanishes on an open dense set 
and hence vanishes everywhere since the kernel is flat. Thus, we have an exact sequence 
\[              0 \rightarrow  \Gr^{1}_{(F_0)_h}(\homs(\sE_0,\sE_0)) \xrightarrow{[\theta_0,-]} \Gr^{-1}_{(F_0)_h}(\homs(\sE_0,\sE_0))  \rightarrow T \rightarrow 0                                         \]
where $T$ is torsion and thus has positive degree. Using Lemma \ref{p:degRank} part (1) and (2), 
\begin{align*}
    \deg((F_0)_h^0(\homs(\sE_0,\sE_0))) &- \deg((F_0)_h^{-2}(\homs(\sE_0,\sE_0)))\\ 
                                    & = -\dfrac{\deg(\Gr^{1}_{(F_0)_h}(\homs(\sE_0,\sE_0))) - \deg(\Gr^{-1}_{(F_0)_h}(\homs(\sE_0,\sE_0))) }{2}\\ 
                                    &= \deg(T) \ge 0.                                                        
\end{align*}
 \end{proof}

\begin{Theorem}
\label{Lemma Dimension count} 
The dimension of $\Nil_{X,\Spr}^{n}$ at a $k$-point $(\sE_0,h_0,\theta_0,F_0)$ is bounded above by 
$(g-1)( \sum_{i=1}^{n-1} r_i^2 + \sum_{i=1}^{n-1} r_i r_{i-1} + \frac{r_0^2 + r_0}{2}) $ where $r_i$ is the rank of $\Gr^{i}_{C}(E_0)$ 
with $C$ as defined in section \ref{s:nilpFilt}. 
\end{Theorem}

\begin{proof}
We can compute 
\begin{align*}
\dim_{(\sE_0,h,\theta_0,F_0)} \Nil_{X,\Spr}^{n} &= -\chi(P)\\
                                                &= -\chi((F_0)_h^0(\homs(\sE_0,\sE_0))) + \chi((F_0)_h^{-2}(\homs(\sE_0,\sE_0)))\\ 
                                                &\le (g-1)(\rank((F_0)_h^0(\homs(\sE_0,\sE_0))) - \rank((F_0)_h^{-2}(\homs(\sE_0,\sE_0))))\\ 
                                                &= (g-1)(\rank(\Gr^{0}_{(F_0)_h}(\homs(\sE_0,\sE_0))) + \rank(\Gr^{-1}_{(F_0)_h}(\homs(\sE_0,\sE_0)))) \\ 
                                                &\leq (g-1)  (\frac{r_0^2 + r_0}{2}    +\sum_{i=1}^{n-1} r_i^2 + \sum_{i=1}^{n-1} r_i r_{i-1})
\end{align*}
where the first inequality follows from Lemma \ref{Lemma degrees is upped bounded} and the others from 
\ref{c:even} and \ref{c:odd}. 
\end{proof}

%% file: equivalence.tex

\section{The Residual Gerbe}

 
\subsection{Involutions and Hermitian modules.}

In this section, 
we will recall some needed definitions and theorems from \cite{Knus}. 

\begin{Definition}
\label{d: involution}
Let $A$ be a $k$-algebra, not necessarily commutative. An involution $\mu$ on $A$
     is an anti-automorphism of $A$ with $\mu\circ \mu = id_A$. 
     We call the pair $(A,\mu)$ a ring (or $k$-algebra) with involution. 
\end{Definition} 

\begin{Lemma}
\label{l:direct product of rings} 
    Let $(A_1,\mu_1)$ and $(A_2,\mu_2)$ be rings with involutions. Then there exists an involution $\mu$ on 
    $A_1 \times A_2$ defined by the morphism 
    \begin{align*}
        \mu : A_1 \times A_2 &\rightarrow A_1 \times A_2 \\ 
              a_1 \times a_2 &\mapsto \mu_1(a_1) \times \mu_2(a_2).           
    \end{align*} 
    We denote this ring with involution as $(A_1,\mu_1)\times( A_2, \mu_2)$.     
\end{Lemma}

\begin{proof}
    This is a straightforward check. 
\end{proof}

\begin{Definition}
    \label{d:hypebolic rings}
    For a ring $A$, we denote by $A^{op}$ the opposite ring which is the same additive group but has 
    the reverse multiplication 
    \[       a^{op}b^{op} = (ba)^{op}                 \] 
    where $a^{op}$ stands for $a$ as an element of $A^{op}$. There is a natural involution $\mu$ on $A \times A^{op}$ 
    defined by 
    \[      \mu (a , b^{op}) =  (b , a^{op}).                                             \]
    The ring with involution $(A \times A^{op}, \mu)$ is defined to be the hyperbolic ring of $A$, and we denote 
    it by $H(A)$.
\end{Definition}
    
\begin{Lemma}
    \label{l:Decomposition of involution of rings}
    Given a semisimple ring with involution $(A,\mu)$, there exists a decomposition of rings with involutions
    \[     (A,\mu) = (A_1 ,\mu^A_1) \times (A_2 , \mu^A_2) \times \hdots \times (A_n ,\mu^A_n) \times (B_1 , \mu^B_1) \times (B_2 , \mu^B_2) \times \hdots \times (B_m , \mu^B_m).  \]
    where $A_i$ are simple rings for $1 \le i \le n$ and $(B_j,\mu^B_j)$ is isomorphic to the hyperbolic ring
    $H(A'_j)$ of some simple ring $A'_j$. 
\end{Lemma} 
    
    \begin{proof}
        This is \cite[Chapter I,1.2.8]{Knus}.
    \end{proof} 

\begin{Notations}
    We denote by $\module(A)$ the category of right $A$-modules. Let $\proj(A)$ denote  the full subcategory of
    $\module(A)$ consisting  of finitely generated projective right modules over $A$ . Let $\proj^r(A)$ denote the full
    subcategory of $\proj(A)$ consisting of finitely generated projective right modules over $A$ of rank $r$. We denote 
    by $\proj_i(A)$ the groupoid core of $\proj(A)$ i.e. the subcategory of $\proj(A)$ with isomorphisms as 
    morphisms.
\end{Notations}

\begin{Notations}
\label{n:bar}
The involution $\mu$ on $A$ allows us to define a left $A$-module structure on any right $A$-module $M$ by 
the equation 
\[     a.m = m.\mu(a)                                  \]
where $a$ and $m$ belong to $A$ and $M$ respectively. We use the notation $\overline{M}$ for the left $A$-module 
corresponding to a right $A$-module $M$.\\
For a module $M$, we denote by $M^\vee$ the dual module 
$\Hom_A(M,A)$. For $M$ a right $A$-module, $\Hom_A(M,A)$ has a natural left module structure given by 
\[     a.f(m) = f(ma)                                  \]  
for $a$ in $A$, $m$ in $M$ and $f$ in $M^\vee$. We denote by $M^*$ the right $A$ module $\overline{M^\vee}$. 
\end{Notations}

\begin{Definition}
\label{d:M to M*}
    We define a contravariant functor
\begin{align*}
    (-)^{*} : \module(A) &\rightarrow \module(A) \\ 
          M      \mapsto M^*    &\text{ and }           f      \mapsto f^*. 
\end{align*}  
where 
\[      f^* = \bar{f}^\vee : N^* \rightarrow M^*                                \] 
is the right $A$-module homomorphism dual to the left $A$-module homomorphism given by 
\[    \overline{f} : \overline{M} \rightarrow \overline{N}.                             \]
The functor $(-)^{*}$ restricts to an endofunctor of both $\proj(A)$ and $\proj^r(A)$, both of which we call $(-)^{*}$ 
by abuse of notation.
\end{Definition}

The purpose of the next lemma is to define a natural isomorphism $\can$.   

\begin{Lemma}
    \label{l:dual bar commutes}
    For $M$ a right $A$-module, the map 
    \[    T_M  :     \overline{M^{\vee}} \to \overline{M}^{\vee}                        \]
    is an isomorphism.
\end{Lemma} 

\begin{proof}
    This is \cite[Chapter I, 2.1.1]{Knus}.
\end{proof}

\begin{Lemma}
\label{l:canonical morphism}
The morphism 
\[   \can_M : M \rightarrow M^{**}                         \]
defined by 
\[    M \stackrel{ev}{\to} M^{\vee\vee} \stackrel{\sim}{\xrightarrow[T_{M^\vee}]{}} M^{**}                         \] 
defines a natural isomorphsim 
\[     \can : id \implies (-)^{*}\circ (-)^{*}                                   \] 
where $id$ is the identity endofunctor of $\module(A)$. We have the identity
\[       (\can_M)^{*} \can_{(M)^{*}} = id            \]

\end{Lemma}  

\begin{proof}
    The first part is in \cite[Chapter I, 2.1.2]{Knus}. The second part is \cite[Chapter I, Proposition 3.1.2]{Knus}.
\end{proof}

\begin{Definition}
    A sesquilinear form $b$ on a right $A$-module $M$ is a biadditive map
    \[      b : M \times M \to A        \]
    such that
    \[ b(xa,ya') = \mu(a)b(x,y)a' \]
    for $x$, $y$ in $M$, and $a$, $a'$ in $A$.
    The set of sesquilinear forms on $M$ is an algebra over the centre $Z$ of $A$, 
    denoted by $\Sesq_{A}(M)$.
\end{Definition}

\begin{Lemma} 
\label{l:adjoint}    
An element $b$ of $\Sesq_A(M)$ defines an $A$-module homomorphism 
\[        h_b : M \rightarrow M^*        \]
defined by  
\[        h_b(x)(y) = b(x,y)                              \]
which we call the adjoint of $b$. 
Furthermore, the map
\[  \psi : Sesq_{A}(M) \to Hom(M,M^{*})          \]
\[          b \mapsto h_{b}                               \]
is an isomorphism of $Z$-modules.
 \end{Lemma}

\begin{proof}
    This is \cite[Chapter I, 2.2]{Knus}.
\end{proof} 

\begin{Definition}
    A sesquilinear form $b$ on $M$ is called nonsingular if the adjoint map $h_{b}$ is an isomorphism. 
\end{Definition} 

\begin{Definition}
   A sesquilinear form on a right $A$-module $M$ is defined to be $\epsilon$-hermitian if it satisfies 
    \[  b(x,y) = \epsilon\mu(b(y,x))                      \] 
    for $\epsilon = \pm 1$.
 \end{Definition}

\begin{Lemma}
\label{l:adjoint hermitian}
    A sesquilinear form on a right $A$-module $M$ is $\epsilon$-hermitian if and only if the adjoint morphism satisfies 
    \[          h_b = \epsilon (h_b)^{*} \circ \can_M .                        \]
\end{Lemma}

\begin{proof}
    This is \cite[Chapter I, 2.3]{Knus}.
\end{proof} 

\begin{Definition}\label{d:herm}
    We define a category $\herm_{\epsilon}(A,\mu)$ whose objects are pairs $(M,b)$ where $M$ is in $\proj(A)$ and $b$ is a nonsingular 
    $\epsilon$-hermitian form on $M$ with respect to the involution $\mu$ on $A$ and the morphisms from objects $(M,b)$ to $(M',b')$ are right 
    $A$-module isomorphisms 
    \[   \phi : M \rightarrow M'              \]
    that satisfy
    \[   b(\phi(x),\phi(y)) = b(x,y)                              \]  
    for all $x$ and $y$ in $M$.
\end{Definition} 

\begin{Notations}
    In view of Lemma \ref{l:adjoint} and Lemma \ref{l:adjoint hermitian}, we sometimes denote an object of $\herm_{\epsilon}(A,\mu)$
    by $(M,h)$ where 
    \[ h : M \rightarrow M^* \]
    is an isomorphism satisfying $h = \epsilon (h)^{*} \circ \can_M$. It is easy to check that the data of a morphism between two objects
    $(M_1,h_1)$ and $(M_2,h_2)$ is the data of a right module isomorphism 
    \[ f : M_1 \rightarrow M_2 \]
    satisfying $h_1 = (f)^{*} h_2 f$.  
\end{Notations}

\begin{Notations}
    We denote by $\herm_{\epsilon}^r(A,\mu)$ the full subcategory of $\herm_{\epsilon}(A,\mu)$ with objects 
    pairs $(M,b)$ where $M$ is in $\proj^r(A)$ and $b$ is a nonsingular 
    $\epsilon$-hermitian form on $M$ with respect to the involution $\mu$ on $A$.
\end{Notations}

\begin{Definition}
    The Jacobson radical $J$ of a ring $A$ is defined to be the intersection of all maximal left ideals of $A$.
\end{Definition}

\begin{Lemma}
    \label{l:Jacobson right max}
    The Jacobson radical $J$ is the intersection of all the maximal right ideals of $A$.
\end{Lemma}

\begin{proof}
    \cite[Chapter I,4.2.2]{Knus}.
\end{proof} 

\begin{Lemma}
\label{l:invol iso max ideal}
    Let $(A,\mu)$ be a ring with involution. If $I$ is a left ideal of $A$, then $\mu(I)$ is a right ideal of $A$ 
    and vice versa. If $I$  is maximal, then so is $\mu(I)$. Furthermore, the induced map of sets
    \[     \mu : \text{\{left maximal ideals of A\} } \rightarrow  \text{\{right maximal ideals of A\} }             \]
    is a bijection.
\end{Lemma}

\begin{proof}
    For $a$ in $A$, $\mu(I).a = \mu(\mu(a).I) \subset \mu(I)$ proving $\mu(I)$ is a right ideal. Suppose $I$ is a maximal
    left ideal. If $\mu(I) \subset J$ with $J$ a proper left ideal of $A$, then $I \subset \mu(J)$ with $\mu(J)$ a
    a proper right ideal of $A$. By maximality of $I$, $I = \mu(J)$ thus implying $\mu(I) = J$. This proves the maximality 
    of $\mu(I)$.\\ 
    The induced map 
    \[     \mu : \text{\{left maximal ideals of A\} } \rightarrow  \text{\{right maximal ideals of A\} }             \]
    is injective since  applying $\mu$ to $\mu(I) = \mu(J)$ gives $I = J$. It is also surjective since for 
    $J$ a maximal right ideal, $\mu(J)$ is a maximal left ideal that maps to $J$.
\end{proof}

\begin{Corollary}
    Suppose $J$ is the Jacobson radical of a ring with involution $(A,\mu)$. Then, $\mu(J) = J$.
\end{Corollary}

\begin{proof}
    This follows from Lemma \ref{l:Jacobson right max} and Lemma \ref{l:invol iso max ideal}.
\end{proof}

Thus, the involution $\mu$ on $A$ induces an involution $\mu_{A/J}$ on $A/J$ so that the quotient map 
\[      q : A \twoheadrightarrow A/J                     \]
satisfies 
\[      q(\mu(a)) = \mu_{A/J}(q(a)).                                           \]
On the other hand, the quotient map $q$ induces a functor 
\begin{align*}
    Q : \proj(A) &\rightarrow \proj(A/J) \\ 
        M        \mapsto     M \otimes_A A/J &\text{ and }  f        \mapsto     f \otimes_A id_{A/J}. 
\end{align*}

This functor $Q$ satisfies some interesting properties when $A$ is a right artinian ring. 

\begin{Lemma}
\label{l:Q full and surjective}
For $A$ a right artinian ring, the functor 
\[    Q : \proj(A) \rightarrow \proj(A/J)                                 \]
is full, essentially surjective, and conservative.
\end{Lemma}

\begin{proof}
   This follows from \cite[Corollary 2.2]{BDH}.
\end{proof}

Our next aim is to show that this induces a functor from $\herm_{\epsilon}(A,\mu)$ to $\herm_{\epsilon}(A/J,\mu_{A/J})$ that satisfies 
similar properties. We first need a preliminary Lemma.

\begin{Lemma}
    \label{l:nat iso QX to XQ}
    The morphism 
    \[   J_M : M^* \otimes_A A/J \rightarrow (M \otimes_A A/J)^*  \] 
    defined by 
    \[    J_M (f \otimes a_1)(m \otimes a_2) =   \mu_{A/J}(a_1)q(f(m))a_2                              \]
    induces a natural isomorphism 
    \[      J : Q \circ (-)^{*}_A \implies (-)^{*}_{A/J} \circ Q                                    \]
    where $(-)^{*}_A$ and $(-)^{*}_{A/J}$ are the corresponding endofuctors of $\proj(A)$ and $\proj(A/J)$. 
     Furthermore, the following diagram commutes 
    \[\begin{tikzcd}
                & M \arrow[ld, "\can_M" '] \arrow[r, "Q"]             & M\otimes_{A}A/J \arrow[rd, "\can_{Q(M)}"]     &                          \\
              M^{**} \arrow[rd, "Q" ']  &                                                  &                                                       & \left( M\otimes_{A}A/J\right)^{**}\arrow[ld, "(-)^* _A \circ J_M"'] \\
                & M^{**}\otimes_{A}A/J  \arrow[r, "J_{M^*}" ', shorten <= -.5em ] & \left(M^{*}\otimes_{A}A/J \right)         &                         
        \end{tikzcd}\]
\end{Lemma} 

\begin{proof}
   This is a subcase of \cite[Chapter I, 7.1]{Knus}.
\end{proof}

\begin{Lemma}
\label{l:herm(A) to herm(A/J)}
The functor 
\[         Q : \proj(A) \rightarrow \proj(A/J)          \]
induces a functor 
\[         Q_{\herm} : \herm_{\epsilon}(A,\mu) \rightarrow \herm_{\epsilon}(A/J,\mu_{A/J}).   \]
\end{Lemma}

\begin{proof} 
    We will prove the result for  $\epsilon = 1$ and leave the other case to the reader.  We can define a functor 
    \begin{align*}
        Q_{\herm} : \herm_{1}(A,\mu) &\rightarrow \herm_{1}(A/J,\mu_{A/J})\\  
        (M,h) &\mapsto (Q(M), J_M \circ Q(h)).                                            
    \end{align*}
    We have to check that this is well defined. Firstly, $h$ being an isomorphism implies that $J_M \circ Q(h)$
    is an isomorphism as well since $Q$ is a functor and $J_M$ is an isomorphism. Since $h$ satisfies 
    $h = (h)^{*} \circ can_M$, we have by a diagram chase
    \begin{align*}
        (J_M \circ Q(h))^{*} \circ \can_{Q(M)} &= (Q(h))^{*} \circ (J_M)^{*} \circ \can_{Q(M)} \\ 
                                            &= (Q(h))^{*} \circ J_{(M)^{*}} \circ Q (\can_M)\\  
                                            &= J_M \circ Q((h)^{*}) \circ Q (\can_M) \\ 
                                            &= J_M \circ Q(h)
    \end{align*}        
    where the first equality is by the commutative diagram in Lemma \ref{l:nat iso QX to XQ} and the second equality is by 
    naturality of $J_M$. It follows that $J_{M}\circ Q$ defines a hermitian form by \ref{l:adjoint hermitian}. 
    Lastly, Lemma \ref{l:nat iso QX to XQ} implies that if $f$ is a map of hermitian modules from $(M_1,h_1)$ to $(M_2,h_2)$,
    $Q(f)$ is a map of hermitian modules from $(Q(M_1), j_{M_1} \circ Q(h_1))$ to $(Q(M_2), j_{M_2} \circ Q_{h_2})$. 
\end{proof} 

We are finally ready to prove a lemma analogous to Lemma \ref{l:Q full and surjective} for $Q_{\herm}$.

\begin{Lemma}
    \label{l:Q herm surjective lifts iso}
    For $A$ a right artinian ring, the functor 
    \[    Q_{\herm} : \herm_{\epsilon}(A,\mu) \rightarrow \herm_{\epsilon}(A/J,\mu_{A/J})                                 \]
    is essentially surjective and full.
\end{Lemma}

\begin{proof}
    We will prove the result for $\epsilon = 1$, the other case is left as an exercise. Let $(M_J,h_J)$ be an object of $\herm_{1}(A/J,\mu_{A/J})$. Applying Lemma \ref{l:Q full and surjective}, there exists
    $M$ in $\proj(A)$ such that $Q(M) \cong M_J$.  By the same lemma, there exists a right $A$-module isomorphism
    \[       h' : M  \rightarrow (M)^{*}                         \] 
    such that $Q(h') = J_M^{-1} \circ h_J$. We define 
    \[         h = \frac{h' + (h')^{*} \circ \can_M}{2}.                                     \]   
    We compute 
    \begin{align*}
        ((h')^{*} \circ \can_M)^{*} \circ \can_M &= (\can_M)^{*} \circ ((h')^{*})^{*} \circ \can_M \\ 
                                           &= (\can_M)^{*} \circ \can_{(M)^{*}} \circ h \\ 
                                           &= h.
    \end{align*}
    where the second equality follows from the naturality of $\can$ and the last follows from property of $\can$
    in Lemma \ref{l:canonical morphism}. This proves that 
    \begin{align*}
        (h)^{*} \circ \can_M &= \frac{(h')^{*} \circ \can_M +  ((h')^{*} \circ \can_M)^{*} \circ \can_M }{2} \\ 
                          &= \frac{h' + (h')^{*} \circ \can_M }{2} \\ 
                          &= h.
    \end{align*}
    and thus $(M,h)$ is hermitian. Then $h$ also satisfies  
    \begin{align*}
        Q_{\herm}(h) &= \frac{ Q_{\herm} (h') + Q_{\herm}((h')^{*} \circ \can_M)}{2} \\ 
                     &=  \frac{h_J + (h_J)^{*} \circ   \can_{Q(M)}}{2} \\ 
                     &= h_J
    \end{align*} 
    where the second equality follows from the computation in Lemma \ref{l:herm(A) to herm(A/J)} and the last follows 
    from the fact that $h_J$ is hermitian. Since $j_M \circ Q(h) = Q_{\herm}(h) = h_J$ is an isomorphism, $Q(h)$ is 
    isomorphism, and thus so is $h$ by Lemma \ref{l:Q full and surjective}.\\ 
    Now, suppose $f_J$ is an  isomorphism
    \[    f_J : (M_J , h^M_J) \rightarrow (N_J,h^N_J)                         \]
    in $\herm_1(A/J,\mu_{A/J})$. Let $(M,h^M)$ and $(N,h^N)$ be objects in $\herm_1(A,\mu)$ that satisfies 
    \[    Q_{\herm}(M,h^M) = (M_J , h^M_J) \quad \text{and} \quad  Q_{\herm}(N,h^N) = (N_J , h^N_J).              \] 
    Applying Lemma \ref{l:Q full and surjective}, there exists an isomorphism 
    \[   f' : M \rightarrow N              \]  
    satisfying $Q(f') = f_J$. We define $f = \frac{f' + (h^N)^{-1}(f')^{*-1}h^M}{2}$. By a similar
    diagram chase as before, it follows that $f = (h^N)^{-1}(f)^{*-1}h^M$ and $Q_{\herm}(f) = Q(f) = f_J$. 
    Since $Q(f)$ is an isomorphism, $f$ is an isomorphism by Lemma \ref{l:Q full and surjective}. 
\end{proof} 

Given a right 
$A$-module $M$, we can use the projection 
\[   H(A) = A \times A^{op} \rightarrow A   \] 
to endow $M$ with a right $H(A)$ module structure. Since $H(A)$ is a ring with involution, it makes sense to talk about the 
right $H(A)$-module $(M)^{*}$. We can also define a right 
 $H(A)$-module homomorphism 
\[    h : M \oplus (M)^{*} \rightarrow (M)^{*} \oplus (M)^{**}                                                \]  
by the matrix 
\begin{equation}
    h  =  \begin{pmatrix}
        0  & id \\ 
        \can_M & 0.
    \end{pmatrix}
\end{equation}
It is easy to check that this defines a hermitian nonsingular sesquilinear form on $M \oplus (M)^{*}$. We denote the 
corresponding object of $\herm_1(H(A))$ by $H(M)$. Furthermore, given any isomorphism 
\[    f : M_1 \rightarrow M_2              \] 
in $\proj(A)$, we can define a morphism 
\[      H(f)  : M_1 \oplus (M_1)^{*} \rightarrow M_2 \oplus (M_2)^{*}                          \] 
by the matrix 
\begin{equation}
   H(f)  =  \begin{pmatrix}
        f  & 0\\ 
        0 & (f)^{*-1}
    \end{pmatrix}.
\end{equation}
It is again a computation to check that 
\[         H(f) : H(M_1) \rightarrow H(M_2)            \]
is actually a morphism in $\herm_1(H(A))$. This allows us to define a functor 
\begin{align*}
    H : \proj_{iso}(A) &\rightarrow  \herm_1(H(A)) \\ 
    M \mapsto  H(M)    &\text{ and }  f \mapsto H(f)     
\end{align*} 

\begin{Lemma}
    \label{l:Herm modules over Hyperbolic rings}
    The functor 
     \[
      H : \proj_{iso}(A) \rightarrow \herm(H(A))   
    \]
    is an equivalence of categories.
\end{Lemma}

\begin{proof}
    Given an object $(M,h)$ in $\herm_1(H(A))$, we define right $H(A)$-modules
    \[    M_1 = M.(1,0) \quad \text{and} \quad M_2 = M.(0,1)            \] 
    Since $(1,0)$ and $(0,1)$ in $A \times A^{op}$ are idempotents satisfying $(1,0).(0,1)$ = 0, we have a decomposition
    \[    M_1 \oplus M_2 \stackrel{\sim}{\to} M                          \]
    Under this identification,we can express $h$ as a matrix 
    \[  h = \begin{pmatrix}
        h_{11} & h_{12} \\ 
        h_{21} & h_{22}.
    \end{pmatrix}   \quad \text{where} \quad h_{ij} : M_j \to (M_i)^{*}.                   \] 
    Furthermore for all elements $m$ and $n$ of $M$, 
    \[    h(m.(1,0))(n.(1,0)) = (0,1)h(m)(n.(1,0)) = 0.                     \] 
    This proves that the morphism $h_{11}$ is 0, and so is $h_{22}$ by a similar computation. Since $h$ is an isomorphism, $h_{12}$ and $h_{21}$ are isomorphisms
    and $h_{21} = (h_{12})^{*} \can_{M_2}$ since $h$ is hermitian. Thus, we have an isomorphism 
    \[  id \oplus h_{21} :  M_1 \oplus M_2 \rightarrow M_1 \oplus (M_1)^{*}                            \]
    and under this identification, the hermitian form $h$ is 
    $$ \begin{pmatrix} 0 & id \\ 
                    \can_{M_1} & 0. \end{pmatrix}$$ 
    thus proving that $M \cong H(M_1)$.\\ 
    If $f$ is an isomorphism 
    \[    f : H(M) \rightarrow H(N),   \] 
    we can define morphisms 
    \begin{align*}
        f_1 : M &\rightarrow N \quad &\text{and} \quad f_2 : (M)^{*} &\rightarrow (N)^{*} \\ 
             m  &\mapsto  f(m).(1,0) \quad &\text{and} \quad m &\rightarrow f(m).(0,1).
    \end{align*}
    We then have a decomposition 
    \[ f = \begin{pmatrix}
        f_1 & 0 \\ 
        0 & f_2 
    \end{pmatrix}.\]
    The fact that $f$ is a morphism in $\herm_1(A)$ implies that $f_2 = (f_1)^{*-1}$, thus proving that $f = H(f_1)$.
    This proves that $H$ is full. Faithfulness of $H$ is clear.    
\end{proof}

\begin{Corollary}
\label{c: hyperbolic mod iso extension} 
Let $(M,h^M)$ be $(N,h^N)$ be objects of $\herm^r_1(H(A))$ for a right artinian ring $A$. Then 
there is an isomorphism of hermitian modules $$
(M,h^{M})\to (N,h^{N}). 
$$

\end{Corollary} 
\begin{proof}
    As the modules have the same rank, $Q(M)$ and $Q(N)$ are isomorphic. Since $Q$ is conservative, \ref{l:Q full and surjective}, there is 
    an isomorphism $M\to N$. 
    The result follows from the lemma. 
\end{proof} 

\begin{Lemma}
    \label{l:herm over direct product of rings} 
        Let $(A_1,\mu_1)$ and $(A_2,\mu_2)$ be rings with involutions. Let $(M,h)$ be in $\herm_1((A_1,\mu_1) \times (A_2, \mu_2))$. 
        Then $M$ has a decomposition 
        \[ (M,h) \cong (M_1 \oplus M_2, h_1 \oplus h_2) \]
        where $(M_1,h_1)$ and $(M_2,h_2)$ are objects of $\herm_1(A_1,\mu_1)$ and $\herm_1(A_2,\mu_2)$ respectively.     
\end{Lemma}
\begin{proof}
    Let $e_1$ and $e_2$ be identities of $A_1$ and $A_2$ respectively. Similarly, as in the proof of 
    Lemma \ref{l:Herm modules over Hyperbolic rings},  the decomposition holds true for
   \begin{align*}
        M_1 = Me_1    \quad &\quad M_2 = Me_2 \\ 
        h_1 = he_1    \quad &\quad h_2 = he_2. 
    \end{align*}    
\end{proof}

\begin{Lemma}
\label{l:simple iso extension}
Let $(M,h^M)$ be $(N,h^N)$ be objects of $\herm^r_1(A,\mu)$ for a simple $k$-algebra $A$ with a $k$-algebra involution $\mu$.
 Then there exists a
finite field extension $L$ of $k$ such that $(M_{L},h^M_{L})$ and $(N_{L},h^N_{L})$ are isomorphic as objects of  
$\herm^r_1(A_{L},\mu_{L})$.
\end{Lemma} 

\begin{proof}
    Let $L$ be the splitting field of $A$. Then $A_{L} \cong \End_{L}(V)$ for a vector space $V$ over $L$. Let 
    $M$ be a simple module over $A_{L}$. Any element $x$ of ${L}$ defines a homomorphism 
    \[      m_x : M \rightarrow M                               \]
    given by the right  multiplication by the element $x$ on $M$. By Schur's Lemma, the corresponding morphism 
    \[            L \rightarrow \End_{A_{L}}(M), \quad \quad x \mapsto m_x                                \]
    is an isomorphism. By \cite[Chapter I, 9.6]{Knus}, $M$ has a $\epsilon$-hermitian form $h$ for 
    $\epsilon = \pm 1$. Then $L$ can be given the structure of involution  by the map 
    \[    L\rightarrow L ,  \quad \quad     m_x \mapsto  h^{-1} (m_x)^{*} h.                   \]   
    The above map is just the identity map by the following computation 
    \begin{align*}
        h^{-1}(m_x)^{*}h &= \epsilon h^{-1} (m_x)^{*}(h)^{*} \can \\ 
                      &=  \epsilon h^{-1} (hm_x)^{*} \can      \\ 
                      &= \epsilon h^{-1} ((m_x)^{*} h)^{*} \can \\ 
                     &= \epsilon h^{-1} (h)^{*} (m_x)^{*} \can   \\ 
                      &= \can^{-1} ((m_x)^{*})^{*} \can \\ 
                     &= m_x.  
    \end{align*} 
    where $hm_{x} = (m_x)^{*}h$ since $\mu_{L}$ preserves $L$. Now by \cite[Chapter I, 9.6.1]{Knus}, the category 
    $\herm_1(A_{L},\mu_{L})$ is equivalent to the category $\herm_{\epsilon}(L,id)$. Since
    there exists an unique object of rank $r$ in $\herm_{\epsilon}(L,id)$ for any $r$ $\in$ $\mathbb{N}$,  
    $(M_{L},h^M_{L})$ and $(N_{L},h^N_{L})$ are isomorphic.
\end{proof}

 
\subsection{Pushforwards of symplectic forms.}

\begin{Notations}
    Let $X$ be a projective scheme over a field $k$. For any finite field extension $k \subset K$, we denote by 
$X_{K}$ the scheme $X \otimes_k K$. Let $\pi_K$ be the canonical projection,
\[     \pi_K : X_{K}  \rightarrow X.                            \]
If $K \subset K'$ is an extension of fields, we denote by 
\[      \pi_{KK'} : X_{K'} \rightarrow X_K                                \]
the canonical projection.
\end{Notations}

Since $k$ is a field of characteristic $0$, the extension $k \subset K$ is separable. The trace morphism 
\[         \Tr : K \rightarrow k                               \]
induces a nonlinear bilinear nondegenerate form. 

\begin{Lemma}
\label{l:trace nondegenerate}
The $k$-bilinear form 
\[        (x,y) \mapsto \Tr(xy)                \]
is nondegenerate.
\end{Lemma}
\begin{proof}
This is \cite[Theorem 5.2]{Lang}.    
\end{proof}

\begin{Lemma}
\label{l:NondegeneracyPushforward}
\begin{enumerate}
    \item There is a natural $\mathcal{O}_X$ linear  isomorphism 
      \[         (\pi_K)_* \cO_{X_K} \stackrel{\sim}{\to} \cO_{X} \otimes_k K                                 . \] 
    \item There is a natural $\mathcal{O}_X$ linear isomorphism  
       \[       \cO_X \otimes_k \Hom_k(K,k)  \stackrel{\sim}{\to} \homsheaf_{\cO_X}((\pi_K)_* \cO_{X_K} , \cO_X).               \]
    \item The composition 
     \[       (\pi_K)_* \cO_{X_K} \stackrel{\sim}{\to} \cO_{X} \otimes_k K  \stackrel{id \otimes_k \Tr}{\rightarrow} \cO_X \otimes_k \Hom_k(K,k)  \stackrel{\sim}{\to} \homsheaf_{\cO_X}((\pi_K)_* \cO_{X_K} , \cO_X).    \]
        is an $\cO_X$-linear isomorphism.
    \end{enumerate}
    In the last statement, the map denoted $\Tr$ is the isomorphism $K\to K^{\vee}=\Hom_{k}(K,k)$ induced by the non-degenerate trace pairing. 
\end{Lemma} 

\begin{proof}
On an open  set $U$ of $X$, the first map is just the map 
\[     \Gamma( U,\cO_{X} \otimes_k K) \rightarrow \Gamma(U, \cO_X) \otimes_k K.                             \] 
Using the first isomorphism, the second map can be defined on an open set $U$ as 
\begin{align*}
    \alpha : \cO_U \otimes_k \Hom_k(K,k) &\rightarrow \homsheaf_{\cO_U}(\cO_{U} \otimes_k K , \cO_U) \\ 
                 \alpha(u \otimes f)(u' \otimes x) &= uu'f(x) 
\end{align*}
It is enough to check that this is an isomorphism over an affine open. This is \cite[Proposition 2.10]{Eisenbud}.
The last map is composition of $\cO_X$ linear isomorphisms.
\end{proof}

\begin{Lemma}
    \label{l:affinePushforward} 
Let $f$ be an affine morphism of schemes
\[f : Y' \rightarrow Y.\]
Given quasicoherent sheaves $\sE$ and $\sF$ on $Y'$ then the natural morphism 
\[     f_* \homsheaf_{\cO_{Y'}}(\sE,\sF) \rightarrow \homsheaf_{f_*(\cO_{Y'})}(f_* \sE, f_* \sF)                                 \] 
is a $\cO_Y$ linear isomorphism.
\end{Lemma} 

\begin{proof}
    This is straightforward. 
\end{proof}

\begin{Lemma}
    \label{l:trace iso} 
Let $\sF$ be a locally free sheaf over $X_K$ of finite rank. Then there is a  $\cO_X$-linear isomorphism
\[         (\pi_K)_*(\homsheaf_{\cO_{X_K}}(\sF,\cO_{X_K})) \rightarrow \homsheaf_{\cO_X}((\pi_K)_* \sF, \cO_X).       \]
 \end{Lemma} 

\begin{proof}
By Lemma \ref{l:affinePushforward}, we have a $\cO_X$-linear isomorphism 
\begin{align*}
    (\pi_K)_*(\homsheaf_{\cO_{X_K}}(\sF,\cO_{X_K})) &\rightarrow  \homsheaf_{(\pi_K)_*\cO_{X_K}}((\pi_K)_*\sF,(\pi_K)_*\cO_{X_K}). 
\end{align*}
      Using Part $(1)$ of Lemma \ref{l:NondegeneracyPushforward} and the trace map, we have a morphism
    \[         (\pi_K)_* \cO_{X_K} \stackrel{\sim}{\to} \cO_{X} \otimes_k K  \xrightarrow[]{id \otimes \Tr }  \cO_{X} .\]                                         
Postcomposing with this morphism, we get a morphism 
\[   \homsheaf_{(\pi_K)_*\cO_{X_K}}((\pi_K)_*\sF,(\pi_K)_*\cO_{X_K}) \rightarrow  \homsheaf_{\cO_X}((\pi_K)_*\sF,(\pi_K)_*\cO_{X_K}) \rightarrow  \homsheaf_{\cO_X}((\pi_K)_*\sF,\cO_{X} ).                                                                             \]
On an open set $U$ where $\sF$ is free, this is an isomorphism by Part (3) of Lemma \ref{l:NondegeneracyPushforward}.
\end{proof} 

\begin{Lemma}
    \label{l:symp pushforward} 
Let $(\sF,h)$ be a symplectic vector bundle over $X_K$. Then  
the isomorphsim
$$
 (\pi_{K})_{*}\sF \stackrel{(\pi_{K})_{*}(h)}{\to } (\pi_{K})_{*}\shom(\sF,\cO_{X_K})\stackrel{\sim}{\to }
 \shom_{\cO_X}((\pi_{K})_{*}\sF,\cO_{X})
$$
gives the vector bundle $(\pi_{K})_{*}\sF$ the structure of a symplectic vector bundle. 
\end{Lemma} 

\begin{proof}
Composing the isomorphism in Lemma \ref{l:trace iso} with the isomorphism 
\[        (\pi_K)_* h : (\pi_K)_* \sF \rightarrow  (\pi_K)_*(\homsheaf_{\cO_{X_K}}(\sF,\cO_{X_K}))                         \]
gives an $\cO_{X}$ linear isomorphism 
\[          \bar{h} :   (\pi_K)_* \sF \rightarrow       \homsheaf_{\cO_X}((\pi_K)_* \sF, \cO_X).                        \]
Let $\bar{Q}$ be the corresponding bilinear form 
\[          \bar{Q} : (\pi_K)_* \sF \times (\pi_K)_* \sF \rightarrow \cO_X.                \]
We need to check $\bar{Q}$ is alternating. It is enough to check it on an open affine scheme $U = \Spec(R)$ 
such that $(\pi_K)_* \sF$ is free on $\Spec(R \otimes_k K)$. On that open set, the bilinear form has the form 
\begin{align*}
    \bar{Q} : R^n \otimes_k K \times R^n \otimes_k K &\rightarrow R \\  
                \bar{Q}(\alpha \otimes x, \beta \otimes y ) &\mapsto \Tr_R(Q(\alpha \otimes x,\beta \otimes y))
\end{align*} 
where 
\begin{align*}
    \Tr_R : R \otimes_k K &\rightarrow R  \\ 
         r \otimes x &\mapsto r \Tr(x), 
\end{align*}  
and $Q$ is the bilinear form induced by $h$. $\bar{Q}$ is alternating since $Q$ is.
\end{proof}

 
\subsection{Symplectic and Hermitian forms.}

\begin{Notations}
\label{n:involution on endomorphism ring} 
Let $K/k$ be a finite field extension. 
Throughout this subsection 
     $(\sF,h)$ will be a symplectic vector bundle over $X_K$ with corresponding symplectic form $Q$. We will write $h_Q=h$ to symbolise the 
     relationship between $h$ and $Q$. 
     We denote the symplectic 
    form on $(\pi_K)_*\sF$ by $\overline{Q_{\sF}}$ and associated isomorphism
    $$
    \overline{h}: \pi_{K,*}\sF \stackrel{\sim }{\to } \pi_{K,*}\sF^{\vee}. 
    $$
     This induces an involution 
    on the right artinian ring $A = \Hom_{\cO_X}((\pi_K)_* \sF, (\pi_K)_* \sF)$, which we denote by $\mu$.  
\end{Notations} 

\begin{Lemma}
\label{l:iso extension A}
Let $(M,h^M)$ and $(N,h^N)$ be objects of $\herm^r_1(A_L,\mu_L)$ for a field $L$. Then there exists a finite field extension 
$L'$ of $L$ such that $(M_{L'},h_{L'}^M)$ and $(N_{L'},h_{L'}^N)$ are isomorphic as objects of $\herm^r_1(A_{L'},\mu_{L'})$.
\end{Lemma}

\begin{proof}
Lemma \ref{l:Q full and surjective} allows us to reduce to the case $A_L$ is semisimple. Then $(A_L,\mu_L)$ can be written as product
of simple rings with involutions and hyperbolic rings. It's enough to prove for the cases of simple rings with involutions 
and hyperbolic rings separately by Lemma \ref{l:herm over direct product of rings}. The case of hyperbolic rings is proved in 
Corollary \ref{c: hyperbolic mod iso extension}. Furthermore, since the involution $\mu$ preserves $Id$ morphism of $(\pi_K)_* \sF$,
$\mu_L$ preserves the field $L$. Hence, the result follows from Lemma \ref{l:simple iso extension}.
\end{proof}

\begin{definition}
    \label{d:twistedForm}
    An \emph{ $L$-form of $(\sF,h)$} where $k\subseteq L$ is a field extension is a symplectic vector bundle  $(\sE,h)$ on $X_L$ such that there exists a 
    field $L'$ containing both $K$ and $L$ so that 
    $$
    (\sE,h)_{L'}\quad \text{and}\quad (F,h)_{L'}
    $$
    are isomorphic. 
    We will call $L'$ a \emph{splitting field} of the form. 

    Note that
    if the field extensions $K/k$ and $L/k$ are finite then $L'/k$ may be taken to be finite. 
\end{definition}

\begin{Lemma}
    \label{l:symp induces herm}
    Let $k \subset L$ be a finite extension and let $(\sE,h_{\sE})$ $L$-form of $(\sF,h)$. Let $L'/k$ be a finite splitting field 
    of the form. Then 
    \begin{enumerate}
        \item the  $A_L$ module $\Hom(\lkF, \sE)$ is projective of finite rank;
        \item the module is equipped with a non-singular 1-hermitian form. 
    \end{enumerate}
\end{Lemma}

    \begin{proof}
    The first statement is from \cite[Theorem 5.3]{BDH} and it's proof. 

    We define
    \[  b : \Hom(\lkF, \sE) \times \Hom(\lkF, \sE)  \to         A_L         \]
    \[ b(f,g) = \pi_{L}^{*}( \overline{h}^{-1}) \circ f^{\vee} \circ h_{{\sE}} \circ g\]
    It is straightforward that this is linear in $g$. For linearity in $f$, we compute
    \begin{align*}
        b(fa,g) &=   \pi_{L}^{*}(\overline{h}^{-1}) \circ (fa)^{\vee} \circ h_{Q_{\sE}} \circ g \\
                &=   \pi_{L}^{*}(\overline{h}^{-1}) \circ a^{\vee} \circ \pi_{L}^{*}( \overline{h}) \circ \pi_{L}^{*}( \overline{h}^{-1}) \circ f^{\vee} \circ h_{{\sE}} \circ g \\
                &= \mu(a)b(f,g).
    \end{align*}
   The following computation proves that $b$ is hermitian as well. 
    \begin{align*}
        \mu(b(f,g)) &=  \pi_{L}^{*}(\overline{h}^{-1}) \circ (\pi_{L}^{*}(\overline{h}^{-1}) \circ f^{\vee} \circ h_{{\sE}} \circ g)^{\vee} \circ \pi_{L}^{*}( \overline{h})  \\
                       &=   \pi_{L}^{*}(\overline{h}^{-1})  \circ g^{\vee}   \circ h_{Q_{\sE}}^{\vee} \circ f^{\vee \vee} \circ (\pi_{L}^{*}(\overline{h}^{-1})^{\vee} \circ  \pi_{L}^{*}(\overline{h}) \\
                       &= b(g,f),
    \end{align*}
    where we have used 
    \[  h^{\vee} =  -h \circ \can \]
    for a symplectic form $Q$. The only thing left to check is that $b$ is nonsingular.
    It is enough to check that the pullback of the map  
\begin{align*}
    &h_{b} : \Hom(\lkF, \sE) \to \Hom(\lkF, \sE)^*  \\ 
    &h_{b}(f)(g) = \pi_{L}^{*}( \overline{h}^{-1}) \circ f^{\vee} \circ h_{{\sE}} \circ g   
\end{align*}
   by $\pi_{LL'}$ is an isomorphism. Furthermore, it is enough to check that the map 
   \[  \pi_{LL'} h_b^d : \pi_{LL'}^*\Hom(\lkF, \sE)^d \to (\pi_{LL'}^* \Hom(\lkF, \sE)^d)^*  \\                                                                             \]
    is an isomorphism. Now, it follows from the projection formula and the hypothesis on $\sE$ and $\sF$ that  
    $\pi_{LL'}^*\Hom(\lkF, \sE)^d$ is free over $A_{L'}$ of rank 1. Since $A_{L'}$ is an artinian ring,
    it is enough to check that the map $\pi_{LL'}^* h_b^d$ is injective. This follows since  
    $\pi_{L}^{*}(\overline{h}^{-1})$ and $h_{{\sE}}$ are isomorphisms.
\end{proof}

\begin{Lemma}
\label{l:dual and hom}
    Let $M$ and $N$ be right  and left modules over a ring with involution $(A,\mu)$. Suppose further that $M$ is finitely generated projective.
    Then the natural map 
    \[     (M)^{*} \otimes_A  N \rightarrow Hom_{A}(M,\overline{N})               \]
    is a bijection, where $\overline{N}$ is the right module corresponding to the left module $N$.(\ref{n:bar})
\end{Lemma}
  
\begin{proof} To maintain sanity we write $n\cdot a$ for the right multiplication in $\overline{N}$ so that 
    $n\cdot a = \mu(a)n$ where the second term is product in $N$. 
    We define a map 
    \begin{align*}
     \phi:   (M)^* \times N &\rightarrow Hom_{A}(M,\overline{N}) \\
        (f,n)         &\mapsto    (m \mapsto n\cdot f(m)).
    \end{align*}
   One checks that $\phi(fa,n) = \phi(f,an)$ for $a$ in $A$. Hence we get a morphism 
   \[   \phi: (M)^* \otimes_A N \to  Hom_{A}(M,\overline{N}). \]
    This is clearly an isomorphism for $M$ free of finite rank. It follows from projective finitely generated $M$ 
    since it is a summand of a free module.
\end{proof}

\begin{Lemma}
    \label{l:herm induces symp}
    Let $k \subset L$ be a finite extension of fields and let  $(M,b)\in \herm_1(A_L,\mu_L)$. Then the vector bundle  $M \otimes_{A_L} \lkF$ over $X_L$ is equipped with a symplectic form.
\end{Lemma}
\begin{proof}
    Let $U\subseteq X_{L}$ be open. 
    For fixed $m\in M$ and $v\in \lkF(U)$ we define a morphism 
    $$
    \phi(m,v): M \oplus \lkF(U) \to \sO_{X_L} 
    $$
    by 
    $$
    \phi(m,v)(m', v') =  \pi^{*}_{L} \overline{Q_{\sF}}(b(m',m)v,v'). 
    $$
    The map $\phi(m,v)$ is additive in both $m'$ and $v'$. It is bilinear as 
    \begin{align*}
        \phi( m , v) (m'a , v') = \pi_L^*\overline{Q_{\sF}}(\mu(a)b(m',m)v,v') &=\pi_L^*\overline{Q_{\sF}}(b(m',m)v,av') \\ 
                         &=  \phi( m , v) (m' , av').
    \end{align*} 
    Thus it descends to a morphism 
    \[h_{Q_b} :  M \times \lkF \rightarrow (M \otimes_{A_L} \lkF)^{\vee}\]     
    It is obvious that $h_{Q_b}(ma,v)= h_{Q_b}(m,av)$ and hence we get a morphism 
    \[h_{Q_b} :  M \otimes_{A_L} \lkF \rightarrow (M \otimes_{A_L} \lkF)^{\vee}.\]
     $X_L$ linearity of $h_{Q_b}$ is obvious since $X_L$ acts only on $\lkF$. One can also check that $h_{Q_b}$ is the 
     composition of $\cO_X$-linear isomorphisms 
     \begin{align*}
        M \otimes_{A_L} \lkF &\xrightarrow{h_M \otimes h_Q} X(M) \otimes_{A_L} \overline{\Hom_{\cO_{X_L}}(\lkF, \cO_{X_L})} \\ 
                            &\stackrel{\sim}{\to} \Hom_{A_L}(M,\Hom_{\cO_{X_L}}(\lkF, \cO_{X_L})) \\ 
                            &\stackrel{\sim}{\to} \Hom_{\cO_{X_L}}(M \otimes_{A_L} \lkF, \cO_{X_L}).
     \end{align*} 
     where the second isomorphism is as in Lemma \ref{l:dual and hom}. We are left to check that the induced bilinear form 
    \begin{align*}
        Q_{b} : M \otimes_{A_L} \lkF \times M \otimes_{A_L} \lkF &\rightarrow \cO_{X_{L}} \\ 
        Q_{b}(m \otimes v, m' \otimes v') &= \overline{Q_{\sF}}(b(m',m)v,v').
    \end{align*}                
    is alternating. This follows from the computation               
    \begin{align*}
        Q_{b}(m \otimes v, m' \otimes v') &= \overline{Q}_{\sF}(b(m',m)v,v')\\
                                          &= \overline{Q}_{\sF}(v,\mu(b(m,m'))v')\\
                                          &= \overline{Q_{\sF}}(v, b(m,m')v') \\ 
                                          &= - \overline{Q_{\sF}}( b(m,m')v' ,v) \\ 
                                          &=  - Q_{b}(m' \otimes v', m \otimes v).
    \end{align*}
\end{proof}

    \begin{Lemma}
        \label{l:WT = id}
    Let $k \subset L$ be a finite extension of fields and let  $(M,b)$ be in $\herm_1(A_L,\mu_L)$. Then
    \begin{enumerate}
        \item $\Hom_{\cO_{X_L}}(\lkF,M \otimes_{A_L} \lkF)$ is a projective module of finite rank;
        \item $\Hom_{\cO_{X_L}}(\lkF,M \otimes_{A_L} \lkF)$ is equipped with a 1-hermitian form;
        \item there exists an isomorphism
        of hermitian modules 
        \[    M  \stackrel{\sim}{\to}  \Hom_{\cO_{X_L}}(\lkF,M \otimes_{A_L} \lkF). \]
    \end{enumerate}
    \end{Lemma} 
    
    \begin{proof}
    The first statement follows from \cite[Theorem 5.3]{BDH} and it's proof. 

    Recall that $M \otimes_{A_L} \lkF$ has a symplectic form by Lemma \ref{l:herm induces symp} which we denote by $Q$. 
    In turn, $\Hom_{\cO_{X_L}}(\lkF,M \otimes_{A_L} \lkF)$ is equipped 
    with a nonsingular 1-hermitian form by Lemma \ref{l:symp induces herm}. We denote this form by $b_Q$.
    The morphism 
    \begin{align*}
      &F:  M  \rightarrow  \Hom_{\cO_{X_L}}(\lkF,M \otimes_{A_L} \lkF)  \\
        &F(m)(a) = m \otimes a. 
    \end{align*}    
    is an isomorphism of projective modules since this is just the identity morphism under the tensor-hom adjunction
    \begin{align*}      \Hom_{A_L}(M,\Hom_{\cO_{X_L}}&(\lkF,M \otimes_{A_L} \lkF)) \cong \\ 
                     &\Hom_{\cO_{X_L}}(M \otimes_{A_L} \lkF,M \otimes_{A_L} \lkF).                        \end{align*}
    We now compute for $a$ in $\lkF$,
    \begin{align*}
        b_Q(F(m_1), F(m_2))(a) &= \pi_{L}^{*}( \overline{h}^{-1}) \circ F(m_1)^{\vee} \circ h_{Q} \circ F(m_2)(a) \\ 
                               &=  \pi_{L}^{*}(\overline{h}^{-1}) \circ F(m_1)^\vee Q(m_2 \otimes a, -) \\ 
                               &=  \pi_{L}^{*}(\overline{h}^{-1}) \overline{Q_{\sF}}(b(m_1,m_2)a, - ) \\ 
                               &= b(m_1,m_2)a.
    \end{align*}

    \end{proof}
    
    \begin{Lemma}
        \label{l:TW = id}
        Let $k \subset L$ be a finite extension of fields and let $(\sE,Q_{\sE})$ be a $L$-form of $(\sF,h)$. 
        Then
        \begin{enumerate}
            \item the  $A_L$ module $\Hom(\lkF, \sE)$ is projective of finite rank;
            \item $\Hom_{\cO_{X_L}}(\lkF,\sE) \otimes_{A_L}  \lkF$ has a form making it into a symplectic vector bundle over $X_L$;
            \item there exists an isomorphism of symplectic vector bundles 
    \[        \Hom_{\cO_{X_L}}(\lkF,\sE) \otimes_{A_L}  \lkF \stackrel{\sim}{\to} \sE.                        \]
        \end{enumerate}
    \end{Lemma} 
    
    \begin{proof}
    The first statement is from \cite[Theorem 5.3]{BDH}. 

    There is a nonsingular 1-hermitian form on $\Hom(\lkF, \sE)$ by Lemma \ref{l:symp induces herm}, which we call $b$. 
    Lemma \ref{l:herm induces symp} implies that  $\Hom_{\cO_{X_L}}(\lkF,\sE) \otimes_{A_L}  \lkF$ is a symplectic vector bundle over $X_L$,
    and we denote the symplectic form by $Q_b$. There exists a morphism of sheaves 
        \begin{align*}
          G:  \Hom_{\cO_{X_L}}(\lkF,\sE) \otimes_{A_L}  \lkF &\rightarrow \sE   \\ 
               G(f \otimes \alpha_{U}) = f|_U(\alpha_U).
        \end{align*}
    This is an isomorphism of vector bundles since it is an isomorphism after pullback to a field $L'$. We need 
    to prove that this is an isomorphism of symplectic vector bundles. For $x,y$ in $\lkF$ and $f,g$ in 
    $\Hom_{\cO_{X_L}}(\lkF,\sE)$, 
    \begin{align*}
        Q_b(g \otimes x, f \otimes y)            &= \pi_L^*\overline{Q}_{\sF}(b(f,g)x,y)\\
         &= \pi_L^*\overline{Q}_{\sF}(\pi_{L}^{*}(\overline{h}^{-1}) \circ f^{\vee} \circ h_{Q_{\sE}} \circ g(x) ,y)\\
         &= f^{\vee} \circ h_{Q_{\sE}} \circ g(x) (y)\\
         &= h_{Q_{\mathcal{E}}} \circ g(x) (f(y))\\
         &= Q_{\sE}(g(x),f(y)).
    \end{align*}
    \end{proof}

     
    \subsection{A categorical equivalence.}
    
    \label{s:categorical}
    
    Let  $\sG$ be a residual gerbe of some point of $\Bun_{X,\Spr}$ with residue field $k(\sG)$. 
    There exists a finite field extension $K$ of $k(\sG)$ such that $\sG(K)$ is nonempty. Let 
    \[         d = [ K : k(\sG)] < \infty   \] 
    As before, we denote by 
    \[            \pi_K :  X_K \rightarrow X_{k(\sG)}                          \]
    Let $\sF$ be a point of $\sG(K)$ and $(A,\mu)$ be the ring $\End((\pi_{K})_*\sF)$ with involution as defined in 
    \ref{n:involution on endomorphism ring}.
    
    \begin{Theorem}
    \label{l:cat equivalence}
    For any field $L$ containing $k(\sG)$, the category $\sG(L)$ is equivalent to the category $\herm_1^{1/d}(A_L)$.
    \end{Theorem}
    
    \begin{proof} 
        Combining \cite[Theorem 5.3]{BDH} and Lemma \ref{l:symp induces herm}, we have a functor
        \begin{align*}
            W : \sG(L) &\rightarrow \herm_1^{1/d}(A_L) \\ 
                \sE    &\mapsto     \Hom_{\cO_{X_L}}(\lkF,\sE).  
    \end{align*}
    
    On the other hand, let $M$ be an object of $\herm_1^{1/d}(A_L)$, then $M \otimes_{A_L} \lkF$ is a symplectic vector 
    bundle over $X_L$ by Lemma \ref{l:herm induces symp}. Moreover, there exists a finite field extension $L'$ of $L$ 
    containing $K$ such that $\Hom_{\cO_{X_{L'}}}(\pi_{L'}^*(\pi_K)_*\sF,\sF \otimes_K L')$ is in $\herm^{1/d}_1(A_{L'})$.
    By Lemma \ref{l:iso extension A}, there exists a finite field extension, which we call $L'$ by abuse of notation such that
    \[     \Hom_{\cO_{X_{L'}}}(\pi_{L'}^*(\pi_K)_*\sF,\sF \otimes_K L') \cong M_{L'}                                    \]
    as objects of $\herm^{1/d}_1(A_{L'})$. Lemma \ref{l:TW = id} proves that 
    \[          \sF \otimes_K L' \cong  M_{L'} \otimes_{A_{L'}} \pi_{L'}^*(\pi_K)_*\sF                                 \]
    as symplectic vector bundles, proving that  $M \otimes_{A_L} \lkF$ is an object of $\sG(L)$. This gives us a functor
    \begin{align*}
         T : \herm^{1/d}(A_L) &\rightarrow \sG(L) \\ 
                M             &\mapsto  M \otimes_{A_L} \lkF
    \end{align*}
    They are mutually inverse by Lemmas \ref{l:WT = id} and \ref{l:TW = id}.
    \end{proof}
    
    \begin{Corollary}
        \label{c:residue gerbe}
    Let $\sG$ be a residue gerbe of $\Bun_{X,\Spr}$. Then, 
    \[        \ed_{k(\sG)}(\sG) = \ed_{k(\sG)}(\herm_1^{1/d} A).                   \]
    \end{Corollary}

%% file: mainResults.tex
\section{The Main Results}
\label{s:mainResults}

We refer the reader to the introduction of this paper where the definition of $\ed_{k}( \mathfrak{X})$ is recalled for an algebraic 
stack  $\mathfrak{X}$. 

Recall that we have fixed a smooth projective geometrically connected curve $X$ over a field $k$ of characteristic 0. The purpose of this section is to 
compute the essential dimension of the stack $\Bun_{X,\Spr}$ of principal $\Spr$-bundles over $X$, see \ref{t:essential} below.

 
\subsection{Fields of moduli.}

\begin{Definition}
\label{d:Levi indecomposable}
A symplectic vector bundle $(\sE,h)$ is defined to be \emph{Levi decomposable} 
if there is a graded symplectic vector bundle $(\sE_{-1}\oplus \sE_{0}\oplus \sE_{1},h_{-1}\oplus h_{0}\oplus h_{1})$
with $\rank \sE_{-1}>0$ and an isomorphism of symplectic vector bundles 
$$
\sE \stackrel{\sim }{\to } \sE_{-1}\oplus \sE_{0}\oplus \sE_{1}. 
$$
\end{Definition} 

\begin{remark}
    This definition can be reinterpreted in terms of the associated principal $\Spr$-bundle to $(\sE,h)$. It means that 
    this principal bundle has reduction of structure group to a Levi subgroup of $\Spr$ associated to a maximal parabolic 
    subgroup.  
    To see this, first recall that the associated principal bundle has functor of points given by
    \begin{eqnarray*}
        \Sch/ X & \longrightarrow & \Sets  \\ 
        (f:S\to X) & \longmapsto & {\rm Isom}((f^{*}E,f^{*}h), (\sO_{S}^{2r}, h)), 
    \end{eqnarray*}
    where $(\sO_{S}^{2r}, h)$ is the trivial symplectic vector bundle of rank $2r$. 

   An isomorphism $\phi :(f^{*}E,f^{*}h)\to  (\sO_{S}^{2r}, h) $ identifies $\sE_{-1}$ with an isotropic subbundle of the trivial bundle
   which in turn produces a reduction of structure group to ${\rm Gl}_n \times {\rm Sp}_{2r-2n}$ where $n=\rank \sE_{-1}$. 
\end{remark}

\begin{Lemma}
\label{l:Aut unipotent}
Let $(\sE,h)$ be a symplectic vector bundle over a smooth projective curve $X$ which is Levi indecomposable. Then,  
$\frac{\Aut^0(\sE,h)}{\pm Id}$ is an unipotent algebraic group.
\end{Lemma} 

\begin{proof} 
This is \cite[Proposition 2.4]{BBN}.
\end{proof} 

\begin{Lemma} 
\label{l:0 on fib implies 0}
Let $s$ be a section of a vector bundle $\sE$ over a curve $X$. Suppose further that the induced section $s|_x$ of the 
fiber  $\sE_x/m_x\sE_x$ vanishes for all closed points $x$ of $X$. Then $s = 0$.  
\end{Lemma}

\begin{proof}
   Since the question is local, we can assume that $X = \Spec(R)$ where $R$ is an integral domain finitely generated over a field,
   and $s$ is an element of a finitely generated free module $M$ over $R$. 
   The condition on $s$ implies that $s$ is in $mM$ for all maximal ideal $m$
   of $R$. Since $R$ is a Jacobson ring and integral domain, intersection of all maximal ideals is $0$. This implies that $s$ is 0.
\end{proof} 

Let $K \supset k$ be an extension of fields and let $\sE$ be a vector bundle over $X_K$.
 We denote by $ T_{Id} \Aut^0(\sE,h)$ the tangent space at identity of $\Aut^0(\sE,h)$. Just to recall, 
 it is the vector space of morphisms 
 \[     \Spec(K[\epsilon]) \rightarrow \Aut^0(\sE,h),   \quad \quad \epsilon^2 = 0                        \]
 such that 
 \[       \Spec(K) \mapsto Id.                           \]
In particular,  the tangent space $T_{Id} \Aut^0(\sE,h)$ can be identified with a subset of
$\Aut^0( \sE_{K[\epsilon]},h_{K[\epsilon]})$ consisting of automorphisms 
\[    g :   ( \sE_{K[\epsilon]},h_{K[\epsilon]}) \rightarrow ( \sE_{K[\epsilon]},h_{K[\epsilon]})         \]
such that the following diagram commute 
\[\begin{tikzcd} 
    \sE_{K[\epsilon]} \arrow[r,two heads] \arrow[d,"g"] &\sE \arrow[d,"id"]\\ 
    \sE_{K[\epsilon]}  \arrow[r,two heads]              &\sE.
\end{tikzcd}\]
With this identification, we have the following isomorphism.

\begin{Lemma}
\label{l:tangent space of Aut}
Let $\sE$ be a symplectic vector bundle over $X_K$. Then there exists a vector space isomorphism 
\begin{align*}
    \phi: \Homs(\sE,\sE) &\rightarrow T_{Id} \Aut^0(\sE,h)\\ 
    \theta &\mapsto Id+\theta\epsilon. 
\end{align*}      
\end{Lemma} 

\begin{proof}
    Given $\theta$ in $\Homs(\sE,\sE)$, clearly 
    \[     Id + \theta \epsilon :  \sE \otimes_K K[\epsilon] \rightarrow \sE \otimes_K K[\epsilon]                    \]
    is an endomorphism of the vector bundle $\sE$. Furthermore it is an isomorphism since $(Id + \theta \epsilon)(Id - \theta \epsilon) = Id$. 
    Furthermore,
    \begin{align*}
    Q_{h_{K[\epsilon]}}((Id + \theta \epsilon)(x_1 + y_1 \epsilon), (Id + \theta \epsilon)(x_2 + y_2 \epsilon)) &- Q_{h_{K[\epsilon]}}(x_1 + y_1 \epsilon, x_2 + y_2 \epsilon) \\ 
                                                                &= \epsilon(Q_h(\theta(x_1),x_2) + Q_h(x_1,\theta(x_2))) \\ 
                                                                &= 0.
    \end{align*}
    proving the well-definedness of the map $\phi$. On the other hand, given an element $g$ of $T_{Id} \Aut^0(\sE,h)$, 
    $g - Id$ defines an element of $\Hom(\sE,\sE)$, which furthermore belongs to $\Homs(E,E)$ by running the above calculation 
    backwards. This defines a morphism 
    \begin{align*}
        \psi: T_{Id} \Aut^0(\sE,h) &\rightarrow \Homs(\sE,\sE) \\ 
        g &\mapsto g -Id. 
    \end{align*}      
    It is clear that $\phi$ and $\psi$ are inverses of each other.
\end{proof}

\begin{Lemma}
\label{l:differentiation of restriction map}
Let $x$ be a point of $X$. Then differential at $Id$ to the restriction  morphism   
\begin{align*}
    \Aut^0(\sE,h) &\rightarrow \Aut^0(\sE_x/m_x) \\
     g            &\mapsto     g|_x
\end{align*}
is the restriction morphism 
\begin{align*}
    \Homs(\sE,h) &\rightarrow \Hom(\sE_x/m_x,\sE_x/m_x) \\
     \theta           &\mapsto     \theta|_x.
\end{align*}
\end{Lemma}
\begin{proof}
    This is straightforward from the definition.
\end{proof}

\begin{Lemma} 
    \label{l:nilpotency of adjoint}
    Suppose that $k$ is algebraically closed. 
    Let $(\sE,h)$ be a Levi indecomposable symplectic vector bundle of rank $r$ over a curve $X$. Then every element of 
    $\Homs(\sE,\sE)$ is nilpotent.
\end{Lemma}

\begin{proof}
Let $\theta$ be in $\Homs(E,E)$. Let $x$ be a closed point of $X$ such that cokernel of the induced map on stalks at $x$
is flat. The restriction map 
\begin{align*}
    \Aut^0(\sE,h) &\rightarrow \Aut^0(\sE_x/m_x) \\
     g            &\mapsto     g|_x
\end{align*}
defines a finite dimensional representation of $\Aut^0(\sE,h)$.
Since $\frac{\Aut^0(\sE,h)}{\pm Id}$ is unipotent by Lemma \ref{l:Aut unipotent},
every element in the image of the differential at $Id$ of the restriction morphism
\begin{align*}
    \frac{\Aut^0(\sE,h)}{\pm Id} &\rightarrow \Aut^0(\sE_x/m_x) \\
     g            &\mapsto     g|_x
\end{align*} 
is nilpotent by \cite[pg 304, 14.31]{Milne}. Hence the restriction of $\theta$ to the fibre at $x$ is nilpotent by Lemma \ref{l:differentiation of restriction map}. 
We are done by Lemma \ref{l:0 on fib implies 0}. 
\end{proof} 

\subsection{Transcendence degree bounds.}

For this section, we fix a curve $X$ of genus $g$ over a base field $k$.

\begin{Notations}
    For $(\sE,h)$ a $k$-point of $\Bun_{X, \Sp_{2r}}$, we denote by $\overline{\{\sE\}}$ the closure of $\sE$ in $\Bun_{X, \Sp_{2r}}$ 
    (refer \cite[\href{https://stacks.math.columbia.edu/tag/0508}{Tag 0508}]{stacks-project}). 
\end{Notations}

\begin{Notations}
For $(\sE,h)$ be a $k$-point of $\Bun_{X, \Sp_{2r}}$, we denote by $\Nilscheme_s(\sE)$ be the closed subscheme of $\Homs(\sE,\sE)$ consisting of nilpotent elements.
Let $n$ be the order of nilpotency and let $r_i$ (where $1-n\le i \le n-1$) be the rank of $\Gr^i_C(\sE)$ for a general element
in $\Nilscheme_s(\sE)$, with $C$ as defined in section \ref{s:nilpFilt}. We denote by $\Nilscheme_s^g(\sE)$ the open dense subscheme
of $\Nilscheme_s(\sE)$ consisting of elements $\theta$ that has  order of nilpotency $n$ and satisfies 
 $\rank(\Gr^i_C(\sE)) = r_i$. Let $\N$ be the closure of points $(\sE,h,\theta,F)$ in $\Nil_{X,\Spr}^{n}$ 
such that $\theta$ is in $\Nilscheme_s^g(\sE)$. Let $\tN$ be the closure of points $(\sE,h,\theta)$ in $\tNil_{X,\Spr}^{n}$
such that $\theta$ is in $\Nilscheme_s^g(\sE)$.
\end{Notations}

\begin{Lemma}
\label{l:dim Nilsn(E)}
Suppose $k$ is algebraically closed. Moreover, let $n$ be the order of nilpotency and
let $r_i$ be the rank of $\Gr^i_C(\sE)$ for an element $\theta$ in $\Nilscheme_s^g(\sE)$, with $C$ as defined in section \ref{s:nilpFilt}.
Then,
\[      \dim_{\sE}\overline{\{\sE\}} + \dim_k (\Nilscheme_s^n(\sE)) \le (g-1)( \sum_{i=1}^{n-1} r_i^2 + \sum_{i=1}^{n-1} r_i r_{i-1} + \frac{r_0^2 + r_0}{2}) \]                                 
\end{Lemma}

\begin{proof}
We choose a smooth surjective morphism from a scheme $U$ to $\overline{\{\sE\}}$. Let
\[   
\begin{tikzcd}
 \N \arrow[r, "\Fgt"]  & \overline{\{\sE\}}         \\
 V \arrow[u]  \arrow[r, "f"]   &   U \arrow[u]
    \end{tikzcd}
\]
be cartesian. Since $\overline{\{\sE\}}$ is reduced by definition, $U$ is by \cite[\href{https://stacks.math.columbia.edu/tag/04YF}{Tag 04YF}]{stacks-project}
and \cite[\href{https://stacks.math.columbia.edu/tag/04YH}{Tag 04YH}]{stacks-project}. Furthermore, we can assume that $U$ is irreducible 
by choosing a connected component. Let $\eta$ be the generic point of $U$. It clearly maps to the generic point $\sE$ 
of $\overline{\{\sE\}}$. Since $V \rightarrow U$ is locally of finite type, we can apply the generic flatness theorem 
\cite[\href{https://stacks.math.columbia.edu/tag/06QS}{Tag 06QS}]{stacks-project} to get that the map $f$ is flat at all points of 
$f^{-1} \eta$. By \cite[\href{https://stacks.math.columbia.edu/tag/04NR}{Tag 04NR}]{stacks-project}, 
\[     \dim_{\eta} U_0 + \dim_v f^{-1} \eta \le \dim_v f^{-1}(U_0)                                      \] 
for all $v$ in $f^{-1}(\eta)$. Now, both the squares in the diagram
\[
\begin{tikzcd}
    \N \arrow[r]                     & \tN \arrow[r]                             & \overline{\{\sE\}} \\
    f^{-1} \eta \arrow[r] \arrow[u] & \Nilscheme_s^g(\sE) \arrow[r] \arrow[u] & \eta \arrow[u]    
    \end{tikzcd} \]
are cartesian and the bottom horizontal map is surjective by Lemma \ref{t:bij on points}. This implies 
\[              \dim f^{-1} \eta \le \dim \Nilscheme_s^g(\sE).                      \]
 By \cite[\href{https://stacks.math.columbia.edu/tag/0DRI}{Tag 0DRI}]{stacks-project} 
and \cite[\href{https://stacks.math.columbia.edu/tag/0DRN}{Tag 0DRN}]{stacks-project}, 
\[      \dim_{\sE} \overline{\{\sE\}} + \dim \Nilscheme_s^g(\sE) \le \dim_{(\sE,h,\theta,F)} \N   \le \dim_{(\sE,h,\theta,F)} \Nil_{X,\Spr}^{n}.                            \]
The result now follows Lemma \ref{Lemma Dimension count}.
\end{proof}

\begin{Lemma}
\label{l:Homs and field of moduli}  
Let $K/k$ be a finitely generated field extension.   
Suppose $\sE$ be  $K$-point of $\Bun_{X, \Sp_{2r}}$. Let $k(\sE)$ be the field of moduli of $\sE$. Then 
\[    \dim_{\sE}\overline{\{\sE\}} = \trdeg_k(k(\sE)) - \dim_K \Homs(E,E)                              \]
\end{Lemma}

\begin{proof}
    Let 
    \[      f : U \rightarrow \overline{\{\sE\}}         \]
    be a smooth  morphism from a scheme $U$ whose image contains $\sE$.
    $U$ is reduced by \cite[\href{https://stacks.math.columbia.edu/tag/04YF}{Tag 04YF}]{stacks-project}. We can furthermore 
    choose $U$ to be irreducible and affine as well since  smooth and connected schemes are irreducible. Let  
    \[
    \begin{tikzcd}
    \sG(\sE) \arrow[r]    & \overline{\{\sE\} }       \\
    V \arrow[u] \arrow[r] & U \arrow[u]
\end{tikzcd}\] 
be a cartesian diagram where $\sG(\sE)$ is the residual gerbe at $\sE$
(it exists by \cite[\href{https://stacks.math.columbia.edu/tag/0H22}{Tag 0H22}]{stacks-project}). The map $V \rightarrow U$ 
is a monomorphism (\cite[\href{https://stacks.math.columbia.edu/tag/042P}{Tag 042P}]{stacks-project}) and hence 
separated(\cite[\href{https://stacks.math.columbia.edu/tag/042N}{Tag 042N}]{stacks-project}). This implies that 
$V$ is separated since $U$ is affine. Hence, there exists an open dense subscheme $V_0$ in $V$ by 
\cite[\href{https://stacks.math.columbia.edu/tag/06NH}{Tag 06NH}]{stacks-project}. Now, the generic point $u$ of $U$
maps to $\sE$ in $\overline{\{\sE\}}$ and hence the image of $V$ contains $u$. For any affine open set $V_1$ of $V_0$,
the map 
\[ V_1 \to U\] 
is monomorphism by \cite[\href{https://stacks.math.columbia.edu/tag/042R}{Tag 042R}]{stacks-project} and dominant.
This proves that $V_1$ is integral with the same function field. Now since $\sG(\sE)$ is a gerbe over the 
residue field $k(\sE)$, we can assume that $V_1$ is of finite type over $k(\sE)$. Hence, for $v$ a preimage of 
$u$ in $V_1$,
\[    \dim_u(U) - \dim_{v}(V_1)    = \trdeg_{k}(k(U)) - \trdeg_{k(\sE)}(k(U))  = \trdeg_{k}k(\sE) .                            \]
Since $V_1 \hookrightarrow V$ is smooth, we have that 
\[\dim_v(V_1) = \dim_v(V)   \]
by \cite[\href{https://stacks.math.columbia.edu/tag/0DRI}{Tag 0DRI}]{stacks-project}. By \cite[\href{https://stacks.math.columbia.edu/tag/0DRN}{Tag 0DRN}]{stacks-project},
\[ \dim_{\sE}\overline{\{\sE\} } - \dim_{\sE} \sG(\sE) = \trdeg_{k}k(\sE) \]
Since $\sG(\sE)$ is $B\Aut(\sE)$ after base changing to an algebraically closed field 
\[    \dim_{\sE} \sG(\sE) = - \dim \Homs(\sE,\sE)                  \]
by Lemma \ref{l:tangent space of Aut}. This finishes off the proof.
\end{proof}

\begin{Lemma}
\label{l:Field of Moduli for indecomposable bundles over algebraically closed field}
Let $X$ be a genus $g\ge 2$ curve over $k$. Let $K \supset k$ be an algebraically closed field and let $\sE$ be a Levi indecomposable symplectic vector bundle over $X_K$.
Let $\theta$ be a  general element of $\Nilscheme_s(\sE)$. Then, we have 
\[      \trdeg_k(k(\sE)) \le (g-1)( \sum_{i=1}^{n-1} r_i^2 + \sum_{i=1}^{n-1} r_i r_{i-1} + \frac{r_0^2 + r_0}{2}) \] 
where $r_i$ is the rank of $\Gr^i_C(\sE)$ with $C$ as defined in section \ref{s:nilpFilt}.                                                
\end{Lemma} 

\begin{proof} 
By Lemma \ref{l:Homs and field of moduli}, 
\[ \trdeg_k(k(\sE)) =  \dim_{\sE}\overline{\{\sE\}} + \dim_K \Homs(\sE,\sE).\] 
Since $\Homs(\sE,\sE) = \Nilscheme_s(\sE)$ by Lemma \ref{l:nilpotency of adjoint}, we are done by Lemma 
\ref{l:dim Nilsn(E)}. 
\end{proof} 

\begin{Corollary}
\label{c:transcendence indecomposable}
    Let $X$ be a genus $g\ge 2$ curve over $k$. Let $K \supset k$ be an algebraically closed field and let $(\sE,h)$ be a Levi indecomposable symplectic vector bundle of rank $2r$  over $X_K$.  Then, we have 
    \[      \trdeg_k(k(\sE)) \le (g-1)( \dim \Spr) \] 
\end{Corollary}

\begin{proof}
    We have $2r=\sum_{i=1-n}^{n-1}r_i=r_0+2\sum_{i=1}^{n-1}r_i$. Also $\dim\Spr = \frac{r^2+r}{2}$. The result follows from the above by an easy calculation. 
\end{proof}

\begin{Corollary}
\label{c:transcendece symplectic}
    Let $X$ be a genus $g\ge 2$ curve over $k$. Let $K \supset$ be an algebraically closed field and let $(\sE,h)$ be a 
    symplectic vector bundle of rank $2r$  over $X_K$.  Then, we have 
    \[      \trdeg_k(k(\sE)) \le (g-1)( \dim \Spr) \] 
\end{Corollary}

\begin{proof}
We can decompose 
\[     \sE = \sE_{-1} \oplus \sE_0 \oplus \sE_{1}                   \]
such that $\sE_0$ is an indecomposable symplectic vector bundle, $\sE_{-1}$ and $\sE_1$ are dual vector bundles. The result  
now follows from  \cite[Corollary 6.4]{BDH}. 
\end{proof}

\subsection{Essential dimension calculations.} 

\begin{Lemma}
\label{l:ed less than rank}
Suppose that X is connected and has a $k$-rational point. Let $\sG$ be a residual gerbe of $\Bun_{X,\Spr}$
with residue field $k(\sG)$. Then we have that
\[ \ed_{k(\sG)}(\sG) \le 2r. \]
\end{Lemma} 

\begin{proof}
    Let $K$ be a finite extension of $k$ such that $\sG(K)$ is non empty. Let $\sE$ be the $K$-point of $\sG(K)$. 
    Suppose
    \[  \pi_K : X_K \rightarrow X_{k(\sG)}\]
    be the usual projection map. We have that $\End_{k(\sG)}((\pi_K)_*\sE)$ is finite dimensional over $K$ and hence artinian.
    By Lemma \ref{l:Decomposition of involution of rings},
    we have the decomposition
    \[   \End_{k(\sG)}((\pi_K)_*\sE)  \cong (A_1 ,\mu^A_1) \times (A_2 , \mu^A_2) \times \hdots \times (A_n ,\mu^A_n) \times (B_1 , \mu^B_1) \times (B_2 , \mu^B_2) \times \hdots \times (B_m , \mu^B_m),  \]
    where $J(\sE)$ is the Jacobson radical of $\End_{k(\sG)}((\pi_K)_*\sE)$ , $A_i$ are simple rings for $1 \le i \le n$ and $(B_j,\mu^B_j)$ is isomorphic to the hyperbolic ring
    $H(A'_j)$ of some simple ring $A'_j$. Combining Lemma \ref{l:Q herm surjective lifts iso} and Lemma \ref{l:herm over direct product of rings}
    with Corollary \ref{c:residue gerbe}, we have 
    \[ \ed_{k(\sG)}(\sG) \le  \sum_{i=1}^n \ed_{k(\sG)} \herm_1^{\frac{1}{d}}  (A_i ,\mu^A_i) + \sum_{j=1}^m \ed_{k(\sG)} \herm_1^{\frac{1}{d}}  (B_j ,\mu^B_j).                          \]                 
    By Lemma \ref{l:Herm modules over Hyperbolic rings}, we have 
    \[  \ed_{k(\sG)} \herm_1^{\frac{1}{d}}  (B_i ,\mu^B_i) =  \ed_{k(\sG)} \proj_{iso}^{\frac{1}{d}}(A'_j). \]
    For each $1 \le i \le n$, $1 \le j \le m$, we can express $A_i$ and $A'_j$ as matrix rings
    \[    A_i \cong M_{n_i}(D_i) \quad \text{and} \quad   A'_j \cong M_{m_j}(D'_j)                                        \]
    over division algebras $D_i$ and $D'_j$. By \cite[9.6.1]{Knus}, there exists involutions $\mu_i$ on division rings $D_i$ such that 
    \[     \ed_{k(\sG)} \herm_1^{\frac{1}{d}}  (A_i ,\mu^A_i) =  \ed_{k(\sG)} \herm_{\epsilon_i}^{\frac{n_i}{d}} (D_i,\mu_i).                      \]
    By \cite[6.2.4]{Knus}, every hermitian form on division algebra $D_i$ is diagonalisable, thus implying 
     $\ed_{k(\sG)} \herm_{\epsilon_i}^{\frac{n_i}{d}} (D_i,\mu_i)  \le  \frac{n_i}{d}\ed_{k(\sG)} \herm_{\epsilon_i}^{1} (D_i,\mu_i)$.  
     Since a hermitian form on $D$ is determined by an element of $D$, we have that 
        $$\ed_{k(\sG)} \herm_{\epsilon_i}^{\frac{n_i}{d}} (D_i,\mu_i) \le \frac{n_i}{d} \ed_{k(\sG)} \herm_{\epsilon_i}^{1} (D_i,\mu_i) \le \frac{n_i}{d}\dim_{k(\sG)} D_i. $$   
    Analogously,
    \[\ed_{k(\sG)} \proj_{iso}^{\frac{1}{d}}(A'_j) < \frac{m_j}{d} \dim_{k(\sG)}(D'_j)\]
    by \cite[Corollary 3.7]{BDH}. By \cite[Lemma 4.1]{BDH}, there exists a decomposition of $(\pi_K)_* \sE$ into indecomposable
    vector bundles 
    \[      (\pi_K)_* \sE \cong \oplus_{i=1}^{n} \sE_i^{n_i} \bigoplus \oplus_{j=1}^m ((\sE')_j^{m_j} \oplus (\sE'')_j^{m_j}   )                 \] 
    with 
    \[             \End(\sE_i)/J(\sE_i) \cong D_i \quad \text{and} \quad \End((\sE')_j)/J((\sE')_j) \cong D'_j.                   \]
    By the previous discussion and \cite[Lemma 4.2]{BDH},
    \begin{align*}
        \ed_{k(\sG)}(\sG) &\le \sum_{i=1}^n \frac{n_i}{d}\dim_{k(\sG)} D_i + \sum_{j=1}^m \frac{m_j}{d} \dim_{k(\sG)}(D'_j) \\ 
                          &\le \sum_{i=1}^n \frac{1}{d} \rank(\sE_i^{n_i}) + \sum_{j=1}^m \frac{1}{d} \rank((\sE')_j^{m_j})\\
                          &\le  \frac{1}{d} \rank((\pi_K)_*\sE) \\  
                          &\le 2r.
    \end{align*} 
   
\end{proof}

\begin{Definition}
    A symplectic vector bundle $(\sE,h)$ is defined to be a \emph{simple symplectic vector bundle} if $\Homs(\sE,\sE) = 0$.
\end{Definition}

\begin{Lemma}
\label{l:ed simple}
    Let $X$ be a genus $g$ curve over $k$. Let $K \supset k$ be a field extension and let $(\sE,h)$ be a 
    simple symplectic vector bundle of rank $2r$  over $X_K$.  Then, we have 
    \[      \ed_k(\sE) \le (g-1)( \dim \Spr) + 2^{\nu_2(2r)}          \] 
    where $\nu_2(2r)$ is the highest power of 2 dividing $2r$.
\end{Lemma} 

\begin{proof} 
Since $\Homs(\sE,\sE)$ = 0, the involution $\mu$ on $\Hom(\sE,\sE)$ induced by the symplectic form satisfies 
$ \mu(\theta) = \theta$ for all $\theta$ in $\Hom(\sE,\sE)$. Now, $\Aut(\sE,\sE)$ is the subgroup of elements $\theta$ of $\Hom(E,E)$ that satisfies 
$\mu(\theta) = \theta^{-1}$ and hence $\theta = \theta^{-1}$. This proves that $\Aut(\sE,\sE) \cong \mu_2$. In particular, the residual 
gerbe $\sG(\sE)$ of $\sE$ in $\Bun_{X,\Spr}$ is a $\mu_2$-gerbe. The index of this gerbe is a power of 2 and divides $2r$ 
by \cite[Proposition 3.1.2.1 (ii)]{Lieb}. We have 
\[       \ed_k(\sG) (\sG(\sE)) \le 2^{\nu_2(2r)}        \]
by \cite[Theorem 5.4]{BRV}. Hence,
\[      \ed_k(\sE) =   \trdeg_k(k(\sE)) + \ed_k(\sG) (\sG(\sE))        \le (g-1)( \dim \Spr) + 2^{\nu_2(2r)}          \] 
by Corollary \ref{c:transcendece symplectic}.
\end{proof} 

\begin{Lemma}
    \label{l:ed not simple}
        Let $X$ be a genus $g \ge 2$ curve over $k$. Let $K \supset k$ be a field extension and let $(\sE,h)$ be a non-simple
         symplectic vector bundle of rank $2r$  over $X_K$.  Then, we have 
        \[      \ed_k(\sE) \le (g-1)( \dim \Spr) + 1          \] 
\end{Lemma}  

\begin{proof}  
Since $k(\sE) = k(\sE \otimes_k L)$ for any field extension, we can assume that $K$ is algebraically closed for computing 
$\trdeg_k(k(\sE))$.  If $\sE$ is decomposable, then there exists a decomposition of $\sE$ into 
symplectic bundles
\[     (\sE,h) = (\sE_1,h_1) \oplus (\sE_2,h_2).                                \]
of rank $r_1$ and $r_2$ respectively
Hence, 
\begin{align*}
    \trdeg_k(k(\sE)) &\le \trdeg_k(k(\sE_1)) + \trdeg_k(k(\sE_2)) \\ 
                     & \le (g-1)(  \frac{r_1^2 + r_1 + r_2^2 + r_2 }{2} ) \\ 
                     & \le  (g-1)(\dim \Spr -2r + 1).
\end{align*}                     
by Corollary \ref{c:transcendece symplectic} and \cite[Lemma 6.5]{BDH}. In the case $\sE$ is indecomposable, there exists 
a non-zero nilpotent morphism in $\Homs(\sE,\sE)$ by Lemma \ref{l:nilpotency of adjoint}. Hence, there exists at least 
two terms in the decomposition in Lemma \ref{l:Field of Moduli for indecomposable bundles over algebraically closed field}. 
A similar computation as above now implies 
\[ \trdeg_k(k(\sE)) \le (g-1)(\dim \Spr -2r + 1)           \]
Now, by Lemma \ref{l:ed less than rank},  
\begin{align*}
    \ed_k(\sE)  &=   \trdeg_k(k(\sE)) + \ed_k(\sG) (\sG(\sE)) \\       
                &\le (g-1)(\dim \Spr -2r + 1) + 2r \\ 
                &\le (g-1)(\dim \Spr) - (g-2)(2r-1) - 2r + 1 + 2r \\ 
                &\le (g-1)(\dim \Spr) + 1.
\end{align*}             
\end{proof}
 
\begin{Corollary}
    \label{c:ed symplectic}
 Let $X$ be a curve of genus $g \ge 2$ over a field $k$ of characteristic 0. Then 
\[      \ed_k(\Bun_{X,\Spr}) \le (g-1)( \dim \Spr) + 2^{\nu_2(2r)}          \] 
where $\nu_2(2r)$ is the highest power of 2 dividing $2r$.  
\end{Corollary}

We refer the reader to page 3 of \cite{coskun} for the definition of a regularly stable 
symplectic vector bundle. There is an open substack 
\[
\Bun_{\Spr}^{rs}\hookrightarrow \Bun_{\Spr}\]
paramerising regularly stable symplectic vector bundles and there is a coarse moduli map
\[
    \Bun_{\Spr}^{rs}\to M_{\Spr}^{rs}
\] 
that makes $\Bun_{\Spr}^{rs}$ into a $\mu_{2}$-gerbe over the coarse moduli space. 
The main theorem of \cite[7.4]{coskun} says that the index of this gerbe is $2^{\nu_2(2r)}$. 

\begin{theorem}\label{t:essential}
    Let $X$ be a curve of genus $g \ge 2$ over a field $k$ of characteristic 0. Then 
\[      \ed_k(\Bun_{X,\Spr}) = (g-1)( \dim \Spr) + 2^{\nu_2(2r)}          \] 
where $\nu_2(2r)$ is the highest power of 2 dividing $2r$.  
\end{theorem}

\begin{proof}
    We have already proved that 
    \[      \ed_k(\Bun_{X,\Spr}) \le (g-1)( \dim \Spr) + 2^{\nu_2(2r)}.          \] 
    For the reverse inequality we consider the generic gerbe of 
    \[
    \Bun_{\Spr}^{rs}\to M_{\Spr}^{rs}. 
\] 
By deformation theory one has $\dim M_{\Spr}^{rs}=(g-1)\dim \Spr $ and hence 
$$
\trdeg_{k}k(M_{\Spr}^{rs}) = (g-1)\dim \Spr. 
$$
The reverse inequality now follows from \cite[Thm. 5.4 ]{BRV} and \cite[7.3]{coskun}. 
\end{proof}